\documentclass[10pt,a4paper,reqno]{RAMArticle} 
\usepackage[left=2.00cm, right=2.00cm, top=2.50cm, bottom=2.50cm]{geometry} 

\crefname{notation}{Notation}{Notations} \crefname{restriction}{Restriction}{Restrictions} \crefname{constraint}{Constraint}{Constraints} \crefname{system}{System}{Systems} \crefname{distillate}{Distillate}{Distillates} 
\newcommand*{\FirstLine}[1]{\textbf{#1}}

\newcommand*{\op}[1]{\M{\mathop{\mathfrak{#1}}\nolimits}} 
\newcommand*{\opim}[4]{\M{\mathop{\ifthenelse{\equal{#4}{}}{\mathfrak{#1}_{#2}^{#3}}{\image{\mathfrak{#1}_{#2}^{#3}}{#4}}}}} 
\newcommand*{\MeanValue}[4][]{\M{\ifthenelse{\equal{#1}{}}{\mathfrak{\ol{I}}\mathord{\left\langle#2\right\rangle}\mathord{\left\lbrack#3,#4\right\rbrack_{\Lambda}}}{\mathfrak{M}\mathord{\left\langle\left. #2\right\rvert #1 \right\rangle}\mathord{\left\lbrack#3,#4\right\rbrack_{\Lambda}}}}} 
\newcommand*{\CasSeq}[2][]{\M{\set{\ldots\ifthenelse{\equal{#2}{}}{\CasBeg[#1]}{#2_{\CasBeg[#1]}}}_{\pre}}} 
\newcommand*{\CasBeg}[1][]{\M{\ifthenelse{\equal{#1}{}}{\mathord{\pm}}{\mathord{#1}}1}} 
\newcommand*{\CasLen}[1][]{\M{\lvert\ifthenelse{\equal{#1}{}}{\pm}{#1}\rvert}} 
 
\newcommand*{\ENorm}[1][]{\M{\lVert\ol{r}_{#1}\rVert_{D}}} 

\newcommand*{ 
\indexminus}{\M{\mathbf{\textbf{\textup{-\!-}}}}} 
\newcommand*{ 
\indexplus}{\M{\mathbf{\textbf{\textup{+}}}}} 
\newcommand*{ 
\indexz}{\M{\mathbf{\grave{z}}}} 
\newcommand*{\Upsilonsub}{\M{\Upsilon\!}} 
\newcommand*{\pre}{\M{\!\prec}} 
\newcommand*{\suc}{\M{\!\succ}} 
\newcommand*{\wosetintable}{\M{{\suc_{\phantom{|}}}\!}} 
\newcommand*{\eps}{\M{\epsilon\,}}

\begin{document} 
\title{Existence, uniqueness and classification of plane waves} 
\author[R A Milton]{Robert Milton} \address{} \email{robert.gomez.milton@gmail.com} \subjclass[2010]{} 
\begin{abstract}
	Existence and uniqueness of solutions is examined for the plane wave problem 
	\begin{align*}
		\pd{\xi}\pd{\xi} \breve{\mathsf{x}} + \pd{\xi} \breve{\mathsf{x}} + \mathop{V^{-2}}\breve{\mathsf{r}}(\breve{\mathsf{x}},\breve{\mathsf{z}}) &= 0 \\
		\mathop{D^{-1}} \pd{\xi}\pd{\xi} \breve{\mathsf{z}} + \pd{\xi} \breve{\mathsf{z}} - \mathop{V^{-2}}\breve{\mathsf{r}}(\breve{\mathsf{x}},\breve{\mathsf{z}}) &= 0 \\
		\pd{\xi}\breve{\mathsf{x}}\to 0 \T{ and } \breve{\mathsf{z}} \to 0 \QT{as} \xi\to &-\infty \\
		\breve{\mathsf{x}}\to 0 \T{ and } \breve{\mathsf{z}}\to 1 \QT{as} \xi\to &+\infty 
	\end{align*}
	where reaction function \M{\breve{\mathsf{r}}} is smooth and strictly positive (irreversible). Rudimentary analytic techniques are used to guarantee a unique plane wave at every wavespeed \M{V>V_*} above some threshold. The result readily extends to cutoff reactions \M{\breve{\mathsf{r}}(\breve{\mathsf{x}}<\breve{\mathsf{X}},\breve{\mathsf{z}})=0}. These results are not novel, but the method of proof is. 
\end{abstract}
\maketitle

\section{Introduction} \label{sec:Int} 
A common feature of natural processes is self-accelerating production, epitomized by population growth, thermal ignition and chemical autocatalysis. The rate \M{\breve{\mathsf{r}}} of irreversible reaction converting fuel (or prey) \M{\breve{\mathsf{z}}} into product (or predator, or heat) \M{\breve{\mathsf{x}}} is contingent upon and enhanced by the presence of \emph{both} species. This enables a local seeding \M{\breve{\mathsf{x}}} of product in a large, initially uniform \M{\breve{\mathsf{z}}=1} reservoir of fuel to form a wave of propagation (or flame), as Fickian diffusion balances locally concentrating reaction: 
\begin{align*}
	\pd{t} \breve{\mathsf{x}} &= \mathop{\breve{\nabla}^2}\breve{\mathsf{x}} \mathbin{+} \breve{\mathsf{r}}(\breve{\mathsf{x}},\breve{\mathsf{z}}) \\
	\pd{t} \breve{\mathsf{z}} &= D^{-1}\mathop{\breve{\nabla}^2}\breve{\mathsf{z}} \mathbin{-} \breve{\mathsf{r}}(\breve{\mathsf{x}},\breve{\mathsf{z}}) 
\end{align*}
Units of concentration and distance can always be chosen to render the reservoir concentration of \M{\breve{\mathsf{z}}} and diffusivity of \M{\breve{\mathsf{x}}} equal to unity. The dimensionless parameter \M{D>0} is the relative diffusivity (Lewis number) of fuel \M{\breve{\mathsf{z}}}.

The archetypal wavefront is one without curvature, uniform along two Euclidean directions. The remaining \M{\breve{\xi}}-direction is supposed to support a plane wave, invariant in the frame \M{(\breve{\xi},t)\cong(\breve{\xi}+V\Delta t,t+\Delta t)} moving at constant speed \M{V>0}. Capturing this symmetry in a travelling wave coordinate \M{\xi\deq V(\breve{\xi}-Vt)} yields 
\begin{align*}
	\pd{\xi}\pd{\xi} \breve{\mathsf{x}} + \pd{\xi} \breve{\mathsf{x}} + \mathop{V^{-2}}\breve{\mathsf{r}}(\breve{\mathsf{x}},\breve{\mathsf{z}}) &= 0 \\
	\mathop{D^{-1}} \pd{\xi}\pd{\xi} \breve{\mathsf{z}} + \pd{\xi} \breve{\mathsf{z}} - \mathop{V^{-2}}\breve{\mathsf{r}}(\breve{\mathsf{x}},\breve{\mathsf{z}}) &= 0 
\end{align*}
subject to 
\begin{align*}
	\pd{\xi}\breve{\mathsf{x}}\to 0 \T{ and } \breve{\mathsf{z}} \to 0 &\QT{as} \xi\to -\infty \QT{(exhausted fuel)} \\
	\breve{\mathsf{x}}\to 0 \T{ and } \breve{\mathsf{z}}\to 1 &\QT{as} \xi\to +\infty \QT{(fresh fuel)} 
\end{align*}

\begin{SCfigure}
	[2.0][tb] 
	\includegraphics[scale=1]{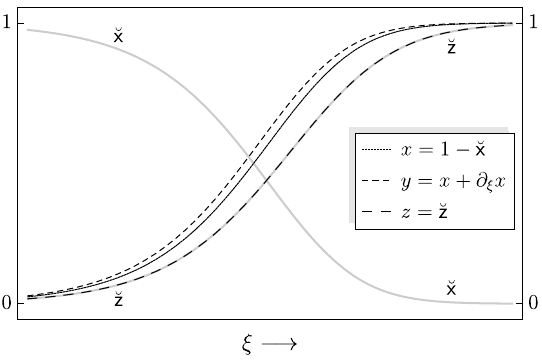} \captionsetup{margin={24pt,8pt}}\caption{Concentration profiles in grey of a plane wave moving in the positive \M{\xi}-direction. Black curves depict the dynamical system formulation \M{\vr{u}\deq(x,y,z)^{\intercal}}, the third component lying precisely atop the grey profile of \M{\breve{\mathsf{z}}}.} 
\label{fig:intro} \end{SCfigure}

\Cref{fig:intro} exemplifies solutions of this equation. The foremost concern is the set of plane wavespeeds \M{\set{V}} supporting such solutions, given reaction \M{\breve{\mathsf{r}}} and diffusivity \M{D}. Which \M{V} is actually observed in practice may depend on the evolution from intial seeding \M{\breve{\mathsf{x}}(\breve{\xi},t=0)}, but this lies beyond (or before) the scope of this study.

Equidiffusive (\M{D=1}) systems have received wide attention since the advent of the Fisher-KPP equations \M{\breve{\mathsf{r}}=\breve{\mathsf{z}}(1-\breve{\mathsf{x}})^K} \citep{Fisher37,Kolmogorov37}, and high activation-energy approximations in flame dynamics \citep{Zeldovich38}. It was demonstrated in the 1970s that any wavespeed \M{V\geq V_*} is possible, provided reaction rate is Lipschitz continuous in both species \citep{Aronson78,Fife77}. Any such \M{\breve{\mathsf{r}}} is endowed with its own minimum wavespeed \M{V_*}, ultimately realized from any strictly local seeding \M{\breve{\mathsf{x}}(\breve{\xi},t=0)}.

Plane wave existence is an essential feature of phase \emph{plane} topology. However, \M{D\ne1} entails a three dimensional phase space eluding elementary topological traps. In the 1980s the general problem was fully solved \citep{Marion85}, by subtle topological deformation of the problem itself instead of its trajectories. This work established plane wavespeeds \M{V\geq V_*} for Lipschitz continuous \M{\breve{\mathsf{r}}} and arbitrary \M{D\in\st{R^+}}, replicating and extending the equiffusive result. It has since been extended to an ever widening problem class \citep{Volpert94}.

This paper presents a basic phase space analysis of Autonomous Dynamical \cref{def:PW:ADS} representing the plane wave problem via 
\begin{gather*}
	\vr{u} \deq (x,y,z)^{\intercal} \deq \left((1 - \breve{\mathsf{x}}), (1 - \pd{\xi}\breve{\mathsf{x}} - \breve{\mathsf{x}}),\breve{\mathsf{z}} \right)^{\intercal} \\
	\Lambda r(\vr{u}) \deq \mathop{V^{-2}}\breve{\mathsf{r}}(\breve{\mathsf{x}},\breve{\mathsf{z}}) 
\end{gather*}
The plane wavespeed is expressed through the lentor (literally, slowness) \M{\Lambda\propto V^{-2}}. It is shown that a unique plane wave exists for any \M{\Lambda\leq\Lambda_*} provided \M{r} is smooth (infinitely differentiable) in both species. This encompasses most physically realistic scenarios, but is weaker than the general result and intends to complement it. Generality extends Brouwer's fixed point theorem by advanced machinery of Leray-Schauder topological degree. In contrast, this paper relies only on the analysts' flint tools of continuity, limit and the mean value theorem (Banach's fixed point also playing its customary Picard-Lindel\"of cameo). Notation supporting these is presented during problem specification in the next Section. Phase space is described in \cref{sec:O}, followed by its entry and exit points in \cref{sec:KP}. The material in these opening Sections is largely standard, primarily providing glossary and database for the Sections that follow. The core argument resides in \crefrange{sec:IP}{sec:UD}, constructing a uniformly continuous, monotonic contrast between trajectories of \M{\Lambda+\eps}. The contrast provides sufficient (though not entire) derivative in \cref{sec:CP}, as the finishing touches are applied to the demonstration. Results and commentary are collected in \cref{sec:Con}.

This primitive approach hopes to provide a secondary, alternative perspective to the work of Marion \emph{et al.}. It is unclear how the two approaches relate to each other, or indeed to the phase plane arguments traditionally applied to equidiffusive systems.

\section{Plane Waves} \label{sec:PW} 
\begin{definition}\label[notation]{def:PW:Not} 
	\FirstLine{Notational conventions} apply as follows. Except miniscule \M{\eps\in\st{R^+}}, constants are written uppercase, variables lowercase. Vectors appear bold, often in terms of the constants 
	\begin{equation*}
		\QM[;]{\vr{0} \deq (0,0,0)^\intercal} \QM[;]{\vr{1} \deq (1,1,1)^\intercal} \QM[;]{\vr{1}_\mathbf{x} \deq (1,0,0)^\intercal} \QM[;]{\vr{1}_\mathbf{y} \deq (0,1,0)^\intercal} \QM{\vr{1}_\mathbf{z} \deq(0,0,1)^\intercal} 
	\end{equation*}
	Compact spaces are overlined, their open interiors not. This endows any continuous function \mapdef{\ol{f}}{\olst{U}}{\st{R}} with maximal open interior \mapdef{f}{\st{U}}{\st{R}} such that \M{\image{\ol{f}}{\olst{U}}=\ol{\image{f}{\st{U}}}}. Here and subsequently, function acts on set to denote the image set. Function as operator dons fraktur typeface, for example projection \op{X} into the \M{yz}-plane 
	\begin{equation*}
		\op{X} \vr{u} \deq (y,z)^\intercal \QT{for all} \vr{u} \deq (x,y,z)^\intercal \in\st{R}^3\deq\open{-\infty}{+\infty}^3 
	\end{equation*}
	However, standard differential operators retain normal typeface and hopelessly redundant overlines may be omitted. When a particular well-order is intended, it subscripts the woset as in \M{\set{\vr{u}_{m}}_{\succ}}.
	
	The class of \M{k\in\st{N}} times continuously differentiable functions on \st{U} is denoted \cont{k}{\st{U}}. The zeroth derivative is \M{f^{(0)}\deq f}, and \cont{k}{\olst{U}} stipulates one-sided continuity of one-sided derivatives on the boundary \M{\olst{U} \setminus \st{U}}. A finite, \st{R}-valued function concept prevails, admitting \M{\ol{\exp} \in \cont{\infty}{\olst{R^-}}} for example while excluding the identity function \M{\ol{1} \notin \cont{0}{\olst{R^-}}}. Continuous \M{f^{(k)}} may thereby lack compactification \M{\ol{f}^{(k)}}, though the latter is surely unique and continuous when existent. This contingency amounts to existence of a uniform norm \M{\norm{f^{(k)}}{\st{U}}} on \cont{0}{\st{U}}, hence deprecated in favour of 
	\begin{equation*}
		\bnorm{\ol{f}^{(k)}}{\olst{U}} \deq \sup \bmodulus{\image{\ol{f}^{(k)}}{\olst{U}}} = \sup \bmodulus{\image{f^{(k)}}{\st{U}}}\in\st{R} \QT{for all} \ol{f}\in\cont{k}{\st{\ol{U}}} \subset \cont{k}{\st{U}} 
	\end{equation*}
	The definitions extend naturally to Fr\'{e}chet derivatives \M{\vr{f}^{(k)}} of vector-valued \mapdef{\vr{f}}{\st{U}}{\st{R}^n} on the understanding that modulus \M{\left|\cdot\right|} signifies euclidean norm on \M{\st{R}^n}. 
	
	Discontinuous derivatives find employment as Lipschitz constants 
	\begin{equation*}
		\bnorm{\olvr{f}_{}^{(k+1)}}{\olst{U}} \deq \sup \bmodulus{\image{\olvr{f}_{}^{(k+1)}}{\olst{U}}} \in\st{R} \QT{for all} \olvr{f}_{}^{}\in\cont[1]{k}{\olst{U}} \subset\cont{k}{\olst{U}} 
	\end{equation*}
	These exist whenever any Lipschitz constant does \citep[pp.\,681]{Schechter97}, as signalled by H\"{o}lder exponent 1 on the function space \cont[1]{k}{\olst{U}}. Again by equivalence of suprema, \cont[1]{k}{\st{U}} is naturally deprecated in favour of the compactified domain. No continuity class \st{F} is itself compact, but then our concern is sequences in \st{U} rather than \M{\st{F}^{(k)}} or \M{\st{F}^{(k,1)}}. 
\end{definition}
\begin{definition}\label{def:PW:Env} 
	\FirstLine{The reaction diffusion environment \M{(\ol{r};D)}} fixes diffusivity \M{D \in \st{R^+}}, and reaction function \ol{r} abiding \crefrange{eq:PW:U def}{eq:PW:r inc z} below. The environment is predetermined once and for all, albeit arbitrarily. This specifies the physical system being studied, habitat for a whole family of plane wave problems.
	
	The reaction function \ol{r} must be non-negative, depending only on the non-negative quantities \M{(1-\ol{x})} and \M{\ol{z}} 
	\begin{equation}\label{eq:PW:U def} 
		\begin{aligned}
			\closed{0}{1}^3 \deqr \olst{U}_{\min} \subseteq \:\:\: &\olst{U} \deq \dom(\ol{r}) \: \subseteq \olst{U}_{\max} \deq \closed{-\infty}{1} \times \olst{R} \times \closed{0}{+\infty} \\
			&\image{\ol{r}}{\olst{U}} \subset \olst{R^+} \\
			&\image{\pd{y}\ol{r}}{\olst{U}} = \set{0} 
		\end{aligned}
	\end{equation}
	in a vital way 
	\begin{equation}\label{eq:PW:ker r} 
		\ker(\ol{r}) \deq \image{\ol{r}_{}^{-1}}{\set{0}} = \setbuilder{\olvr{u}\deq(\ol{x},\ol{y},\ol{z})^\intercal\in \olst{U}}{\ol{z}(1-\ol{x})=0} 
	\end{equation}
	It must be smooth 
	\begin{equation}\label{eq:PW:r smooth} 
		\ol{r} \in \cont{\infty}{\olst{U}} 
	\end{equation}
	and strictly increasing in \M{\ol{z}} 
	\begin{equation}\label{eq:PW:r inc z} 
		\frac{\ol{r}}{\pd{z}\ol{r}} \in \cont{0}{\olst{U}} 
	\end{equation}
	The environment and all it supports is coercively undefined outside \olst{U}. 
\end{definition}
\begin{definition}\label{def:PW:Problem} 
	\FirstLine{A plane wave problem} is specified by Reaction Diffusion Environment~\labelcref{def:PW:Env} and lentor \M{\Lambda \in \st{R^+}}. It consists of Autonomous \cref{def:PW:ADS} subject to Boundary Constraints~\labelcref{def:PW:BC0,def:PW:BC1}. This paper confronts each problem as \M{\Lambda} varies in an unchanging environment. The fixed parameters stay implicit, indexing the problem family by lentor alone. 
\end{definition}
\begin{definition}\label[system]{def:PW:ADS} 
	\FirstLine{The autonomous system} governing plane wave problem \M{\Lambda} is 
	\begin{equation*}
		\pd{\xi} 
		\begin{pmatrix}
			x(\Lambda;\xi)\\y(\Lambda;\xi)\\z(\Lambda;\xi) 
		\end{pmatrix}
		\deqr \pd{\xi}\vr{u}(\Lambda;\xi) = \vr{a}_{\Lambda}(\vr{u}) \deq 
		\begin{pmatrix}
			y(\xi)-x(\xi)\\
			\Lambda r(\vr{u}) \\Dy(\xi)-Dz(\xi) 
		\end{pmatrix}
	\end{equation*}
	where \M{\xi \in \st{R}} is the travelling wave coordinate. The missing arguments on the extreme right should be inferred from the left, as in \M{\vr{u}=\vr{u}(\Lambda;\xi)}. Arguments are dropped in strict order, as in \M{y(\Lambda;\xi)=y(\xi)} never \M{y(\Lambda)}. Arguments will be omitted in this manner when the intended inference is unambiguous, hopefully offering some compromise between fidelity and fluency. 
\end{definition}
\begin{definition}\label{def:PW:BC0} 
	\FirstLine{The original boundary constraint} \M{\olvr{u}(\Lambda;-\infty)=\vr{0}} specifies the exhausted state lingering behind the wavefront. 
\end{definition}
\begin{definition}\label{def:PW:BC1} 
	\FirstLine{The terminal boundary constraint} \M{\olvr{u}(\Lambda;+\infty)=\vr{1}} specifies the fresh state waiting before the wavefront. 
\end{definition}
\begin{proposition}\label{pro:PW:symmetry} 
	\FirstLine{Plane wave symmetry} precludes unique solution: \M{\vr{u}(\Lambda;\xi)} satisfies \crefrange{def:PW:Env}{def:PW:BC1} iff \M{\vr{u}(\Lambda;\xi+\Xi)} does for all \M{\Xi \in \st{R}}. 
\end{proposition}
\begin{definition}\label{def:PW:PW} 
	\FirstLine{A plane wave} is the image \M{\image{\vr{u}}{\st{R}_{\Lambda}}} of any solution to problem \M{\Lambda}, where 
	\begin{equation*}
		\st{R}_{\Lambda} \deq \set{\Lambda} \times \st{R} 
	\end{equation*}
	Subscript lentor will escort any interval in this manner. Assertions like 
	\begin{equation*}
		\ol{f} \in \cont{\infty}{\closed{0}{1}_{\Lambda}} \deq \cont{\infty}{\set{\Lambda}\times\closed{0}{1}} 
	\end{equation*}
	say nothing about \M{\Lambda}-dependence: pedantically following \cref{def:PW:Not}, this stipulates neither-sided continuity of neither-sided derivatives with respect to \M{\Lambda} on \M{\set{\Lambda}}. 
\end{definition}

\section{Orbitals} \label{sec:O}

\begin{SCfigure}
	[2.0][tb] 
	\includegraphics[scale=1]{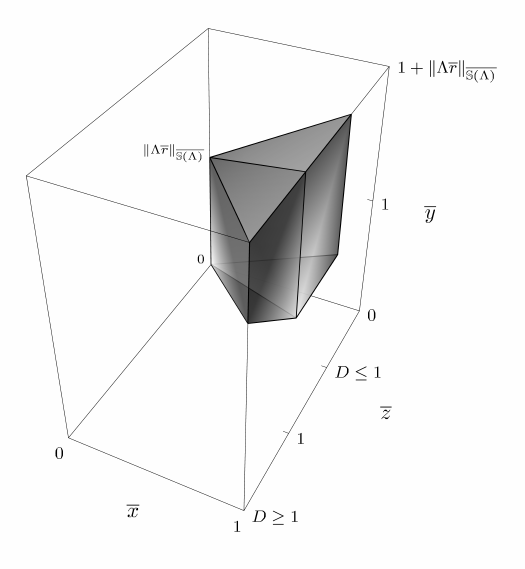} \rule[0pt]{-20pt}{0pt} \captionsetup{margin={0pt,8pt}} \caption{Two illustrative phase spaces for problem \M{\Lambda} in alternative environments. The far triangular prism is germane to environments \M{(\ol{r};D\leq1)}. The near prism is truncated -- the top and bottom triangular faces are not parallel -- by its environment \M{(\ol{r};D\geq1)}. In an equidiffusive environment \M{(\ol{r};D=1)}, the phase space for problem \M{\Lambda} is the plane intersection of the two prisms, on which \M{\ol{z}=\ol{x}}. The equilibrium points \vr{0} and \vr{1}, residing in the base of the phase spaces, have been spotlit.} 
\label{fig:O:PS} \end{SCfigure}

\begin{lemma}\label{lem:O:PS} 
	\FirstLine{The phase space for problem \M{\Lambda}} 
	\begin{equation*}
		\olst[\Lambda]{S} \deq \closed{0}{1} \times \closed{\max(\ol{x},\ol{z})}{\:\ol{x}+\norm{\Lambda\ol{r}}{\olst[\Lambda]{S}}} \times \closed{\ol{x}\min(D,1)}{\:\ol{x}\max(D,1)} 
	\end{equation*}
	depicted in \cref{fig:O:PS} is inescapable when harmlessly subsumed in \olst{U}. It is two-dimensional \M{\ol{z}\in\closed{\ol{x}}{\ol{x}}} in any equidiffusive environment \M{(\ol{r};D=1)}. In which case \textbf{\textup{z}}-components simply duplicate \textbf{\textup{x}}-components, and must be ignored throughout this paper. 
	
	An interior orbital \M{\image{\vr{u}}{\st{R}_{\Lambda}}\subseteq\st{S}(\Lambda)} satisfying Autonomous \cref{def:PW:ADS} is smooth \M{\vr{u}\in\cont{\infty}{\st{R}_{\Lambda}}} and strictly increasing 
	\begin{equation*}
		\image{\pd{\xi}\vr{u}}{\st{R}_{\Lambda}}\subseteq\image{\vr{a}_{\Lambda}}{\st{S}(\Lambda)}\subseteq\st{R^+}\times\st{R^+}\times\st{R^+} 
	\end{equation*}
	Monotonicity remains modulus strict on \olst[\Lambda]{S} outside boundary constraints, so 
	\begin{equation*}
		\set{\T{singular points}} \deq \image{\olvr{a}_{\Lambda}^{-1}}{\set{\vr{0}}} = \set{\vr{0},\vr{1}} 
	\end{equation*}
\end{lemma}
\begin{proof}
	Autonomous \cref{def:PW:ADS} exhibits monotonicity, and rearranges to 
	\begin{align*}
		\pd{\xi}(\ol{y}-\ol{x}) + (\ol{y}-\ol{x}) =\pd{\xi}(\ol{y}-\ol{z}) + D(\ol{y}-\ol{z}) &= \Lambda \ol{r}(\olvr{u}) \\
		\pd{\xi}(\ol{z}-\ol{x}) + D(\ol{z}-\ol{x}) &= (D-1)(\ol{y}-\ol{x}) \\
		\pd{\xi}(\ol{z}-D\ol{x}) + (\ol{z}-D\ol{x}) &= (1-D)\ol{z} 
	\end{align*}
	confirming \M{\olst[\Lambda]{S}\subseteq\olst{U}} a legitimate (inescapable) phase space. Strict monotonicity on \M{\st{S}(\Lambda)} prevents \M{\olst[\Lambda]{S} \setminus \olst{U}_{\min}} from harmfully participating in spurious plane waves, permitting the (perhaps false) subsumption \M{\olst[\Lambda]{S}\subseteq\olst{U}}. 
\end{proof}
\begin{definition}\label{def:O:orbital} 
	\FirstLine{A compact orbital \M{\image{\olvr{v}}{\closed{X_o}{X}_{\Lambda}} \deq \image{\olvr{u}}{\closed{\Xi_o}{\Xi}_{\Lambda}}\subseteq\olst[\Lambda]{S}}} satisfies Autonomous \cref{def:PW:ADS} irrespective of boundary constraints. This curve in three-dimensional space naturally decomposes into its orbital domain 
	\begin{equation*}
		\ol{v}_\mathbf{x}(\Lambda;\ol{x})\deq \ol{x}(\Lambda;\xi) \: \in \: \closed{X_o}{X}_{\Lambda} \deq \set{\Lambda} \times \image{\ol{x}}{\closed{\Xi_o}{\Xi}_{\Lambda}} 
	\end{equation*}
	and orbital projection (discarding \M{\ol{v}_{\mathbf{z}}^{}} in equidiffusive environments \M{(\ol{r};D=1)}) 
	\begin{equation*}
		\begin{pmatrix}
			\ol{v}_\mathbf{y}(\Lambda;\ol{x}) \\
			\ol{v}_\mathbf{z}(\Lambda;\ol{x}) 
		\end{pmatrix}
		\deq \op{\ol{X}}\olvr{v}(\Lambda;\ol{x}) \deq \op{\ol{X}}\olvr{u}(\Lambda;\xi) \: \in \: \op{\ol{X}}\image{\olvr{v}}{\closed{X_o}{X}_{\Lambda}} \deq \op{\ol{X}}\image{\olvr{u}}{\closed{\Xi_o}{\Xi}_{\Lambda}} 
	\end{equation*}
	The corresponding interior orbital is \M{\image{\vr{v}}{\open{X_o}{X}_{\Lambda}} \deq \image{\vr{u}}{\open{\Xi_o}{\Xi}_{\Lambda}}\subseteq\st{S}(\Lambda)}. Incarcerated in phase space, any orbital is legislated by \cref{lem:O:PS}, and is maximal unless proper subset of another. 
\end{definition}
\begin{lemma}\label[distillate]{lem:O:NADS} 
	\FirstLine{The non-autonomous distillate} governing an orbital is 
	\begin{equation*}
		\pd{x} 
		\begin{pmatrix}
			v_\mathbf{y}(\Lambda;x) \\
			v_\mathbf{z}(\Lambda;x) 
		\end{pmatrix}
		\deq \pd{x} \op{X} \vr{v}(\Lambda;x) = \vr{b}_{\Lambda}(\vr{v}) \deq \frac{1}{(v_\mathbf{y}-x)} 
		\begin{pmatrix}
			\Lambda r(\vr{v}) \\
			Dv_\mathbf{y}-Dv_\mathbf{z} 
		\end{pmatrix}
	\end{equation*}
	subject to 
	\begin{equation*}
		\olvr{v}(\Lambda;X_o)=\olvr{u}(\Lambda;\Xi_o) 
	\end{equation*}
	which ensures \M{\olvr{v}(\Lambda;X)=\olvr{u}(\Lambda;\Xi)}. Orbitals are continuous with smooth interior 
	\begin{align*}
		\olvr{v} &\in \cont{0}{\closed{X_o}{X}_{\Lambda}} \\
		\vr{v} &\in \cont{\infty}{\open{X_o}{X}_{\Lambda}} 
	\end{align*}
	Derivatives of projection component \M{\mathbf{n}\in\set{\mathbf{y},\mathbf{z}}_{\succ}} are uniformly continuous provided they exist at both endpoints 
	\begin{equation*}
		\ol{v}_{\mathbf{n}}^{} \in \cont{k}{\closed{X_o}{X}_{\Lambda}} \QIFF \image{\ol{v}_{\mathbf{n}}^{(k)}}{\set{\Lambda}\times\set{X_o,X}}\subset\st{R}. 
	\end{equation*}
\end{lemma}
\begin{proof}
	Projecting Autonomous \cref{def:PW:ADS} using \M{\pd{\xi}x\pd{x}\vr{v}=\pd{\xi}\vr{u}} yields the non-autonomous distillate. Repeated differentiation secures interior smoothness as \M{\pd{\xi}x>0} throughout \M{\st{S}(\Lambda)}. Uniform continuity is trivial upon compactifying \M{v_\mathbf{n}^{(k)}}, the required limits secure by boundary constraints on \M{\olvr{v}^{(0)}} alone. 
\end{proof}
\begin{definition}\label{def:O:KP} 
	\FirstLine{A kernel point \M{\olvr{U}\deq(\ol{X},\ol{Y},\ol{Z})^\intercal\in\olst[\Lambda]{S}\cap\ker(\ol{r})}} is either original (\M{\ol{Z}=0}, interpreted as \M{\ol{X}=0} iff \M{D=1}), or terminal \M{(\ol{X}=1)}. It is singular iff \M{\olvr{U} = \vr{1}\ol{X}}, otherwise regular. Orbitals are described according to the kernel points they encompass. 
\end{definition}
\begin{lemma}\label{lem:O:CO} 
	\FirstLine{The compact orbit \M{\olst[\Lambda]{V}}} is defined as a maximal, singular original orbital. Each problem \M{\Lambda\in\st{R^+}} engenders a unique compact orbit, ever (singular or regular) terminal \M{\olst[\Lambda]{V}=\image{\olvr{v}}{\closed{0}{1}_{\Lambda}}}. 
\end{lemma}
\begin{proof}
	Fundamentally, this Lemma relies on a globally Lipschitz reaction function 
	\begin{equation*}
		\Lambda \ol{r} \in \cont[1]{0}{\olst[\Lambda]{S}} \QIFF \bnorm{\Lambda\ol{r}_{}^{(1)}}{\olst[\Lambda]{S}}\deq \sup_{\vr{u}\in\st{S}(\Lambda)} \sqrt{\Lambda^2\pd{x}r(\vr{u})^2+\Lambda^2\pd{z}r(\vr{u})^2} \in \st{R} 
	\end{equation*}
	guaranteed by \crefrange{eq:PW:U def}{eq:PW:r smooth} and \cref{lem:O:PS}. The mean value theorem applies to \M{\vr{a}_{\Lambda}\in \cont{\infty}{\olst[\Lambda]{S}}}, furnishing as global Lipschitz constant for Autonomous \cref{def:PW:ADS} 
	\begin{equation*}
		\bnorm{\olvr{a}_{\Lambda}^{(1)}}{\olst[\Lambda]{S}} \deq \sup_{\st{S}(\Lambda)} \: \sup_{\olst{O}}\: \bmodulus{\mathop{\vr{a}_{\Lambda}^{(1)}\lrprm{\vr{u}}}\olvr{o}} \QT{for the unit sphere} \olst{O} \deq \setbuilder{\olvr{o}\in\st{R}^3}{\left|\olvr{o}\right|=1} 
	\end{equation*}
	The Jacobian \M{\olvr{a}_{\Lambda}^{(1)}\lrprm{\olvr{u}}} is calculated explicitly in \cref{def:KP:jacob}, from which Cauchy-Schwarz bounds 
	\begin{equation*}
		\bnorm{\olvr{a}_{\Lambda}^{(1)}}{\olst[\Lambda]{S}} \leq \sqrt{2 + \bnorm{\Lambda\ol{r}_{}^{(1)}}{\olst[\Lambda]{S}}^2 + 2D^2} 
	\end{equation*}
	Picard-Lindel\"of thus secures a unique solution \citep{Arnold92,Teschl12} for Autonomous \cref{def:PW:ADS} initiated at the regular kernel point 
	\begin{equation*}
		\olvr{u}_m(\Lambda;-m) = \vr{1}_\mathbf{y} \bnorm{\Lambda\ol{r}}{\olst[\Lambda]{S}} \: \exp\lrprm{-2m\bnorm{\olvr{a}_{\Lambda}^{(1)}}{\olst[\Lambda]{S}}} \QT{for} m\in\st{Z^+} 
	\end{equation*}
	The maximal orbital thereby constructed finishes at \M{\olvr{u}_m(\Lambda;\Xi_{m})\in\olst[\Lambda]{S} \setminus \st{S}(\Lambda)}. \Cref{lem:O:PS} confirms this orbital is terminal \M{\ol{x}_m(\Lambda;\Xi_{m})=1}, and 
	\begin{equation*}
		\vr{u}_{m+1}(\Lambda;-m) \in \closed{0}{y_{m+1}(\Lambda;-m)}^3 
	\end{equation*}
	By Gronwall's inequality 
	\begin{equation*}
		y_{m+1}(\Lambda;-m) \leq \ol{y}_{m+1}(\Lambda;-m-1) \exp\lrprm{\bnorm{\Lambda\ol{r}_{}^{(1)}}{\olst[\Lambda]{S}}} = \bnorm{\Lambda\ol{r}}{\olst[\Lambda]{S}} \: \exp\lrprm{\bnorm{\Lambda\ol{r}_{}^{(1)}}{\olst[\Lambda]{S}}-2(m+1)\bnorm{\olvr{a}_{\Lambda}^{(1)}}{\olst[\Lambda]{S}}} 
	\end{equation*}
	and again 
	\begin{align*}
		\bmodulus{\vr{u}_{m+1}(\Lambda;\xi)-\vr{u}_{m}(\Lambda;\xi)}^2 &\leq \bmodulus{\vr{u}_{m+1}(-m)-\olvr{u}_{m}(-m)}^2 \exp\lrprm{2(\xi+m)\bnorm{\olvr{a}_{\Lambda}^{(1)}}{\olst[\Lambda]{S}}} \\
		&\leq \left(3y_{m+1}^2(-m)+\ol{y}_{m}^2(-m)\right) \exp\lrprm{2(\xi+m)\bnorm{\olvr{a}_{\Lambda}^{(1)}}{\olst[\Lambda]{S}}} \\
		&\leq \bnorm{\Lambda\ol{r}}{}^2 \left(3\exp\lrprm{2\bnorm{\Lambda\ol{r}_{}^{(1)}}{\olst[\Lambda]{S}}-4\bnorm{\olvr{a}_{\Lambda}^{(1)}}{\olst[\Lambda]{S}}}+1\right) \exp\lrprm{2(\xi-m)\bnorm{\olvr{a}_{\Lambda}^{(1)}}{\olst[\Lambda]{S}}} 
	\end{align*}
	Evidently, the sequence \M{\set{\vr{u}_{m}(\Lambda;\xi)}_{\succ}} converges uniformly once \M{m} exceeds \M{\xi}, as does the sequence \M{\set{\vr{a}_{\Lambda}^{(1)}(\vr{u}_{m}(\Lambda;\xi))}_{\succ}}. Which suffices to construct interior orbit 
	\begin{equation*}
		\st{V}(\Lambda) \deq \image{\vr{u}}{\st{R}_{\Lambda}} \deq \lim\limits_{m\to\infty} \image{\vr{u}_{m}}{\st{R}_{\Lambda}} 
	\end{equation*}
	whose compactification is unique, singular original and (singular or regular) terminal. 
\end{proof}
\begin{definition}\label{def:O:MV} 
	\FirstLine{The orbital mean value \M{\MeanValue[g]{f}{X_o}{X} \in \image{f}{\open{\Xi_o}{\Xi}_{\Lambda}}}} of \M{\ol{f} \in \cont{0}{\closed{\Xi_o}{\Xi}_{\Lambda}}} under the non-decreasing weight \M{\ol{g} \in \cont{1}{\closed{\Xi_o}{\Xi}_{\Lambda}}}, \M{\image{\pd{\xi}\ol{g}}{\closed{\Xi_o}{\Xi}_{\Lambda}} \subset \olst{\st{R^+}}} is 
	\begin{equation*}
		\MeanValue[g]{f}{X_o}{X} \deq \MeanValue[g]{f}{\ol{x}(\Lambda;\Xi_o)}{\ol{x}(\Lambda;\Xi)} \deq \frac{1}{\ol{g}(\Lambda;\Xi)-\ol{g}(\Lambda;\Xi_o)} \int_{\Xi_o}^{\Xi} f(\Lambda;\xi) \pd{\xi}g(\Lambda;\xi) \dd{\xi} 
	\end{equation*}
	Dropping the weight and restoring the overline (for emphasis) abbreviates the orbital integral 
	\begin{equation*}
		\MeanValue{\ol{f}}{X_o}{X} \deq\MeanValue{\ol{f}}{\ol{x}(\Lambda;\Xi_o)}{\ol{x}(\Lambda;\Xi)} \deq \int_{\Xi_o}^{\Xi} f(\Lambda;\xi) \dd{\xi} 
	\end{equation*}
	which is obviously not a mean value in general: perhaps \M{\MeanValue{\ol{f}}{X_o}{X} \notin \image{f}{\open{\Xi_o}{\Xi}_{\Lambda}}}.
	
	The mean value property plays a crucial role in this work, warranting great care. The weights \ol{g} that follow are always admissable at origin \vr{0}, but often diverge at singular terminus \vr{1}. Hence this Definition will usually apply on \M{\closed{\Xi_o}{\Xi}_{\Lambda} \subset \clopen{-\infty}{+\infty}_{\Lambda}}, shrinking to \M{\closed{\Xi_o}{\Xi}_{\Lambda} \subset \st{R}_{\Lambda}} in case \M{f} diverges at origin. This compels our deployment of affinely extended real numbers. 
\end{definition}

\section{Kernel Points} \label{sec:KP} 
\begin{definition}\label{def:KP:asym} 
	\FirstLine{An orbit asymptote \M{\image{\vr{\tilde{v}}}{\closed{(1-\eps)\ol{X}}{(1-\eps)\ol{X}+\eps}_{\Lambda}}\subset\olst[\Lambda]{S}} to kernel point \M{\olvr{U}\in\olst[\Lambda]{V}}} obeys 
	\begin{equation*}
		\vr{\tilde{v}}(\Lambda;\ol{x}) = \olvr{u}(\Lambda;\xi) + o\lrp{\olvr{u}(\Lambda;\xi)-\olvr{U}} \QT{as} \ol{x}(\Lambda;\xi) \to \ol{X} 
	\end{equation*}
	and is said to be \M{\Lambda}-elevated iff the elevation \M{(\tilde{v}_\mathbf{y}(\Lambda;\ol{x})-\ol{x})} varies with \M{\Lambda} at some \M{\ol{x}\in\closed{(1-\eps)\ol{X}}{(1-\eps)\ol{X}+\eps}}. 
\end{definition}
\begin{definition}\label{def:KP:jacob} 
	\FirstLine{The Jacobian of Autonomous \cref{def:PW:ADS}} is the Fr\'{e}chet derivative 
	\begin{equation*}
		\olvr{a}_{\Lambda}^{(1)}\lrprm{\olvr{u}} \deq 
		\begin{pmatrix}
			-1 & 1 & 0 \\
			\Lambda \pd{x} \ol{r}(\olvr{u}) & 0 & \Lambda \pd{z} \ol{r}(\olvr{u}) \\
			0 & D & -D \\
		\end{pmatrix}
	\end{equation*}
	Its eigenvalues and eigenvectors at singular \M{\olvr{u}=\vr{1}\ol{X}\in\set{\vr{1},\vr{0}}} are denoted \M{\set{\Omega_\mathbf{j}(\Lambda;\ol{X}),\vr{1}_\mathbf{j}(\Lambda;\ol{X})}} and obey 
	\begin{equation*}
		\vr{1}_\mathbf{j}(\Lambda;\ol{X}) \deq 
		\begin{pmatrix}
			1 \\
			(1+\Omega_\mathbf{j}) \\
			D(1+\Omega_\mathbf{j})/(D+\Omega_\mathbf{j}) 
		\end{pmatrix}
	\end{equation*}
	subject to \M{\vr{1}_\mathbf{j}(\Lambda;\ol{X}) \deq \vr{1}_\mathbf{z}} iff \M{\Omega_\mathbf{j}(\Lambda;\ol{X})=-D}. The last of the three eigenmodes \M{\mathbf{j}\in\st{J}_{\succ}\deq\set{ 
	\indexminus, 
	\indexplus, 
	\indexz}_{\succ}} must be ignored whenever equidiffusivity forces \M{\ol{z}=\ol{x}}. Besides this, eigenmodes are ordered from gradual to abrupt, imposing \M{ 
	\indexminus\prec 
	\indexplus\prec 
	\indexz} on \M{\st{J}_{\succ}} such that \M{\modulus{\Omega_{ 
	\indexminus}}\leq\modulus{\Omega_{ 
	\indexplus}}}. 
\end{definition}
\begin{definition}\label{def:KP:lin} 
	\FirstLine{The linear asymptote \image{\vec{\vr{u}}}{\closed{\Xi_o}{\Xi}_\Lambda} to singular kernel point \M{\vec{\vr{u}}(\Lambda;\pm\infty)=\vr{1}\ol{X}\in\set{\vr{1},\vr{0}}}} is 
	\begin{equation*}
		\vec{\vr{u}}(\Lambda;\xi) \deq 
		\begin{pmatrix}
			\vec{x}(\Lambda;\xi) \\
			\vec{y}(\Lambda;\xi) \\
			\vec{z}(\Lambda;\xi) 
		\end{pmatrix}
		\deq \vr{1}\ol{X} + \sum_{\mathbf{j}\in\st{J}_{\succ}} \Upsilonsub_\mathbf{j}(\Lambda;\ol{X}) \vr{1}_\mathbf{j}(\Lambda;\ol{X}) \exp(\Omega_\mathbf{j}(\Lambda;\ol{X}) \xi) 
	\end{equation*}
	save the exceptional 
	\begin{equation*}
		\T{Replace }\vr{1}_\mathbf{j}(\Lambda;\ol{X}) \T{ with } 
		\begin{cases}
			\vr{1}\ol{X}-\vr{1}\vec{x}(\Lambda;\xi) &\T{if non-hyperbolic } \Omega_\mathbf{j}=0\\
			E_\mathbf{j}(\Lambda;\ol{X})\vr{1}_\mathbf{j}(\Lambda;\ol{X}) \xi + \vr{E}_\mathbf{j}(\Lambda;\ol{X}) &\T{if hyperbolic degenerate } \Omega_\mathbf{j}=\Omega_\mathbf{k} \neq 0 \T{ for } \mathbf{j}\prec\mathbf{k} 
		\end{cases}
	\end{equation*}
	It is orbital iff \M{\image{\vec{\vr{u}}}{\closed{\Xi_o}{\Xi}_\Lambda}\subset\olst[\Lambda]{S}}. 
\end{definition}
\begin{definition}\label{def:KP:eig} 
	\FirstLine{An eigenspace} is specified by stability \M{(\sgn(\Omega_\mathbf{j}(\Lambda;\ol{X})))} to accommodate linear asymptotes. Descriptions such as centre, degenerate and hyperbolic follow standard usage \citep{Jordan07,Anosov97}. Hyperbolicity is determined by reaction function \ol{r} alone, fixed by environment. The remainder of this Section catalogues every possible eigenspace. Supplementing with nonlinear terms where necessary, each linear asymptote \image{\vec{\vr{u}}}{\closed{\Xi_o}{\Xi}_\Lambda} is used to construct a \M{\Lambda}-elevated orbit asymptote \image{\vr{\tilde{v}}}{\closed{\ol{X}(1-\eps)}{\ol{X}(1-\eps)+\eps}_{\Lambda}}. 
\end{definition}
\begin{lemma}\label{lem:KP:0} 
	\FirstLine{Any singular original asymptote to \M{\vec{\vr{u}}(\Lambda;-\infty)=\vr{1}\ol{X}=\vr{0}}} obeys 
	\begin{equation*}
		\pd{z} \ol{r}(\vr{0}) \geq \pd{x} \ol{r}(\vr{0}) = 0 
	\end{equation*}
	by \cref{eq:PW:U def,eq:PW:ker r}, so its Jacobian eigenvalues are 
	\begin{align*}
		\Omega_{ 
		\indexminus}(\Lambda;0) &\deq -\left.\left(D - \sqrt{D^2 + 4D \Lambda \pd{z} \ol{r}(\vr{0})} \right)\right/2 \\
		\Omega_{ 
		\indexplus}(\Lambda;0) &\deq -\left.\left(D + \sqrt{D^2 + 4D \Lambda \pd{z} \ol{r}(\vr{0})} \right)\right/2 \\
		\Omega_{ 
		\indexz}(\Lambda;0) &\deq -1 
	\end{align*}
	Rejecting the stable eigenspace spanned by \M{\set{\vr{1}_{ 
	\indexplus}(\Lambda;0),\vr{1}_{ 
	\indexz}(\Lambda;0)}} uniquely fixes the route of departure in any problem. This is detailed in \cref{tab:KP:0}, functional form determined by environment alone. The linear asymptote in the centre eigenspace determines \M{\vec{x}(\Lambda;\xi)} via an orbital integral 
	\begin{equation*}
		\vec{\vr{u}}(\Lambda;\xi) = \vr{1} \MeanValue{\Lambda\ol{r}(\vec{\vr{u}})}{0}{\vec{x}} \quad \implies \quad \vec{x}(\Lambda;\xi) = \MeanValue{\Lambda\ol{r}(\vr{1}\vec{x})}{0}{\vec{x}} 
	\end{equation*}
\end{lemma}

\begin{table}
	\begin{equation*}
		\begin{array}{|c|c|c|c|c|}
			\hline \rule{0pt}{12pt}\T{\textbf{Environment \M{\ol{r}}}} & \T{\textbf{Eigenspace}} & \vec{\vr{u}}(\Lambda;\xi)-\vr{0} & \T{\textbf{Orbit Asymptote }}\vr{\tilde{v}}(\Lambda;\vec{x}) & \set{\Upsilonsub_{ 
			\indexminus},\Upsilonsub_{ 
			\indexplus},\Upsilonsub_{ 
			\indexz}}_\wosetintable \\
			\hline \hline \rule{0pt}{10pt} \T{hyperbolic} & \T{unstable} & \Upsilonsub_{ 
			\indexminus} \vr{1}_{ 
			\indexminus} \exp(\Omega_{ 
			\indexminus}\xi) & \vec{\vr{u}}(\Lambda;\xi) & \set{+,0,0}_\wosetintable \\
			\pd{z} \ol{r}(\vr{0})>0 & & & & \phantom{\set{}_{\suc_|}} \\
			\hline \rule{0pt}{10pt} \T{non-hyperbolic} & \T{centre} & \vr{1} \MeanValue{\Lambda\ol{r}(\vec{\vr{u}})}{0}{\vec{x}} & \vec{\vr{u}}(\Lambda;\xi) & \set{1,0,0}_\wosetintable \\
			\pd{z} \ol{r}(\vr{0})=0 & & & {}+\left(\vr{1}_\mathbf{y}+(1-D^{-1})\vr{1}_\mathbf{z}\right) \Lambda\ol{r}(\vec{\vr{u}}) & \phantom{\set{}_{\suc_|}} \\
			\hline 
		\end{array}
	\end{equation*}
	\vskip\belowdisplayskip\caption{Departure from the singularity \vr{0}. The final column indicates coefficient values, (\M{+}) denoting a strictly positive unknown. The centre eigenspace incorporates nonlinear supplements for \M{\Lambda}-elevation.} 
\label{tab:KP:0} \end{table}

\begin{corollary}\label{cor:KP:0 smooth} 
	\FirstLine{The compact orbit for \M{\Lambda\in\st{R^+}} originates smoothly \M{\olvr{v}\in\cont{\infty}{\clopen{0}{1}_{\Lambda}}}}. 
\end{corollary}
\begin{corollary}\label{cor:KP:0 orbital} 
	\FirstLine{The compactification of any plane wave is a singular terminal orbit} because the linear asymptote \M{\image{\vec{\vr{u}}}{\closed{-\infty}{\Xi}_{\Lambda}}} is evidently orbital for some \M{\Xi\in\st{R}^-}. 
\end{corollary}
\begin{definition}\label{def:KP:class} 
	\FirstLine{Compact orbits are classified by terminal elevation \M{(\vec{v}_{\mathbf{y}}^{}(\Lambda;\vec{x})-\vec{x})\sim(\ol{v}_{\mathbf{y}}^{}(\Lambda;1)-1)}} 
	\begin{align*}
		\QT{upper class U} &\iff \quad \pd{x}\ol{v}_{\mathbf{y}}^{}(\Lambda;1)\leq 0 \leq \sgn(\ol{v}_{\mathbf{y}}^{}(\Lambda;1)-1) \\
		\QT{middle class M} &\iff \quad \sgn(\ol{v}_{\mathbf{y}}^{}(\Lambda;1)-1) \leq \pd{x}\ol{v}_{\mathbf{y}}^{}(\Lambda;1) \leq 1/2 \\
		\QT{lower class L} &\iff \quad \sgn(\ol{v}_{\mathbf{y}}^{}(\Lambda;1)-1) \leq 1/2 \leq \pd{x}\ol{v}_{\mathbf{y}}^{}(\Lambda;1) 
	\end{align*}
	whenever \M{\ol{v}_{\mathbf{y}}^{}\in\cont[]{1}{\closed{0}{1}_\Lambda}}. Classification and the following Lemmas regarding termination are illustrated in \cref{fig:KP:Class}. 
\end{definition}

\begin{SCfigure}
	[2.0][b] 
	\includegraphics[scale=0.8]{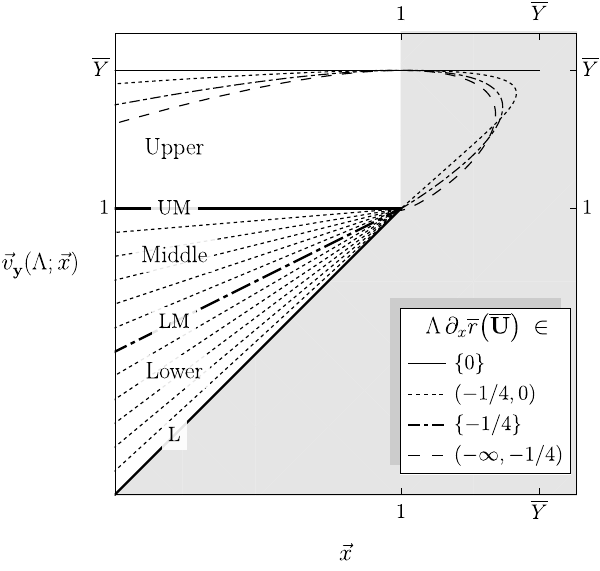} \captionsetup{margin={32pt,8pt}} \caption{Terminal linear asymptotes classified by elevation. A compact orbit is strictly upper class if and only if it is regular terminal, whence we extrapolate into illegal -- grey -- territory following \cref{lem:KP:1reg}. Plane waves are singular terminal, therefore of middle or lower class. Singular asymptotes lying on class boundaries have been thickened for clarity.} 
\label{fig:KP:Class} \end{SCfigure}

\begin{table}
	\begin{equation*}
		\begin{array}{|c|c|c|c|}
			\hline \rule{0pt}{12pt} \T{\textbf{Problem }}\Lambda & \T{\textbf{Eigenspace}} & \vec{\vr{u}}(\Lambda;\xi)-\vr{1}=\vr{\tilde{v}}(\Lambda;\vec{x})-\vr{1} & \set{\Upsilonsub_{ 
			\indexminus},\Upsilonsub_{ 
			\indexplus},\Upsilonsub_{ 
			\indexz}}_\wosetintable \\
			\hline \hline \rule{0pt}{10pt} \Lambda\pd{x} \ol{r}(\vr{1})=-1/4 & \T{stable} & \left((\Upsilonsub_{ 
			\indexminus}\xi+\Upsilonsub_{ 
			\indexplus})\vr{1}_{ 
			\indexplus} - 2 \Upsilonsub_{ 
			\indexminus} \vr{1}_\mathbf{x}\right) \exp(-\xi/2) & \T{LM}\colon \set{-1,\:,-}_\wosetintable \\
			1/2>D & \T{degenerate} & {} - 2D(2D-1)^{-2}\Upsilonsub_{ 
			\indexminus}\vr{1}_\mathbf{z} \exp(-\xi/2) + \Upsilonsub_{ 
			\indexz} \vr{1}_\mathbf{z} \exp(-D\xi)& \T{LM}\colon \set{0,-1,-}_\wosetintable \\
			\hline \rule{0pt}{10pt} \Lambda\pd{x} \ol{r}(\vr{1})=-1/4 & \T{stable} & ((\Upsilonsub_{ 
			\indexminus}\xi+\Upsilonsub_{ 
			\indexplus}) (\vr{1}_\mathbf{x}+\vr{1}_\mathbf{y}/2)-2\Upsilonsub_{ 
			\indexminus}\vr{1}_\mathbf{x})\exp(-\xi/2) & \T{LM}\colon \set{-1,\:,\:}_\wosetintable \\
			D=1/2 & \T{degenerate} & {}+ (\Upsilonsub_{ 
			\indexminus}\xi^2/8+\Upsilonsub_{ 
			\indexplus}\xi/4+\Upsilonsub_{ 
			\indexz})\vr{1}_\mathbf{z}\exp(-\xi/2) & \T{LM}\colon \set{0,-1,\:}_\wosetintable \\
			\hline \rule{0pt}{10pt} \Lambda\pd{x} \ol{r}(\vr{1})=-1/4 & \T{stable} & \left((\Upsilonsub_{ 
			\indexminus}\xi+\Upsilonsub_{ 
			\indexplus})\vr{1}_{ 
			\indexplus} - 2 \Upsilonsub_{ 
			\indexminus} \vr{1}_\mathbf{x}\right) \exp(-\xi/2) & \T{LM}\colon \set{-1,\:,\:}_\wosetintable \\
			D>1/2 & \T{degenerate} & {} - 2D(2D-1)^{-2}\Upsilonsub_{ 
			\indexminus}\vr{1}_\mathbf{z} \exp(-\xi/2) + \Upsilonsub_{ 
			\indexz} \vr{1}_\mathbf{z} \exp(-D\xi) & \T{LM}\colon \set{0,-1,\:}_\wosetintable \\
			\hline \rule{0pt}{10pt} 0>\Lambda\pd{x} \ol{r}(\vr{1})>-1/4 & \T{stable} & \Upsilonsub_{ 
			\indexminus} \vr{1}_{ 
			\indexminus} \exp(\Omega_{ 
			\indexminus}\xi) + \Upsilonsub_{ 
			\indexplus} \vr{1}_{ 
			\indexplus} \exp(\Omega_{ 
			\indexplus}\xi) & \T{L}\colon \set{-1,\:,-}_\wosetintable \\
			-D>\Omega_{ 
			\indexminus}>\Omega_{ 
			\indexplus} & & {}+ \Upsilonsub_{ 
			\indexz} \vr{1}_\mathbf{z} \exp(-D\xi) & \T{M}\colon \set{0,-1,-}_\wosetintable \\
			\hline \rule{0pt}{10pt} 0>\Lambda\pd{x} \ol{r}(\vr{1})>-1/4 & \T{stable} & \Upsilonsub_{ 
			\indexminus}\left(\vr{1}_\mathbf{x}+ (1-D)\vr{1}_\mathbf{y}\right)\exp(-D\xi) + \Upsilonsub_{ 
			\indexplus}\vr{1}_{ 
			\indexplus} \exp(\Omega_{ 
			\indexplus}\xi) & \T{L}\colon \set{-1,\:,\:}_\wosetintable \\
			-D=\Omega_{ 
			\indexminus}>\Omega_{ 
			\indexplus} & \T{degenerate} & {} + \left(\Upsilonsub_{ 
			\indexminus} D(1-D) \xi + \Upsilonsub_{ 
			\indexz}\right) \vr{1}_\mathbf{z} \exp(-D\xi) & \T{M}\colon \set{0,-1,-}_\wosetintable \\
			\hline \rule{0pt}{10pt} 0>\Lambda\pd{x} \ol{r}(\vr{1})>-1/4 & \T{stable} & \Upsilonsub_{ 
			\indexminus} \vr{1}_{ 
			\indexminus} \exp(\Omega_{ 
			\indexminus}\xi) + \Upsilonsub_{ 
			\indexplus} \vr{1}_{ 
			\indexplus} \exp(\Omega_{ 
			\indexplus}\xi) & \T{L}\colon \set{-1,\:,\:}_\wosetintable \\
			\Omega_{ 
			\indexminus}>-D>\Omega_{ 
			\indexplus} & & {}+ \Upsilonsub_{ 
			\indexz} \vr{1}_\mathbf{z} \exp(-D\xi) & \T{M}\colon \set{0,-1,-}_\wosetintable \\
			\hline \rule{0pt}{10pt} 0>\Lambda\pd{x} \ol{r}(\vr{1})>-1/4 & \T{stable} & \Upsilonsub_{ 
			\indexminus}\vr{1}_{ 
			\indexminus} \exp(\Omega_{ 
			\indexminus}\xi) + \Upsilonsub_{ 
			\indexplus}\left(\vr{1}_\mathbf{x}+ (1-D)\vr{1}_\mathbf{y}\right)\exp(-D\xi) & \T{L}\colon \set{-1,\:,\:}_\wosetintable \\
			\Omega_{ 
			\indexminus}>\Omega_{ 
			\indexplus}=-D & \T{degenerate} & {} + \left(\Upsilonsub_{ 
			\indexplus} D(1-D) \xi + \Upsilonsub_{ 
			\indexz}\right) \vr{1}_\mathbf{z} \exp(-D\xi) & \T{M}\colon \set{0,-1,\:}_\wosetintable \\
			\hline \rule{0pt}{10pt} 0>\Lambda\pd{x} \ol{r}(\vr{1})>-1/4 & \T{stable} & \Upsilonsub_{ 
			\indexminus} \vr{1}_{ 
			\indexminus} \exp(\Omega_{ 
			\indexminus}\xi) + \Upsilonsub_{ 
			\indexplus} \vr{1}_{ 
			\indexplus} \exp(\Omega_{ 
			\indexplus}\xi) & \T{L}\colon \set{-1,\:,\:}_\wosetintable \\
			\Omega_{ 
			\indexminus}>\Omega_{ 
			\indexplus}>-D & & {}+ \Upsilonsub_{ 
			\indexz} \vr{1}_\mathbf{z} \exp(-D\xi) & \T{M}\colon \set{0,-1,\:}_\wosetintable \\
			\hline 
		\end{array}
	\end{equation*}
	\vskip\belowdisplayskip\caption{Approach to the hyperbolic singularity \M{\vr{1}} on condition \M{\pd{x} \ol{r}(\vr{1})<0}. The final column indicates class initials and coefficient values, (\M{-}) denoting a strictly negative unknown and (\M{\:}) an unrestricted coefficient. The orbit asymptote \M{\vr{\tilde{v}}(\Lambda;\vec{x})\deq\vec{\vr{u}}(\Lambda;\xi)} is \M{\Lambda}-elevated throughout.} 
\label{tab:KP:1Hyp} \end{table}

\begin{table}
	\begin{tabular}
		{|c|c|c|c|c|} \hline \rule{0pt}{12pt}\textbf{Eigenspace} & \M{\vec{\vr{u}}(\Lambda;\xi)-\vr{1}\ol{X}} & \textbf{Orbit Asymptote }\M{\vr{\tilde{v}}(\Lambda;\vec{x})} & \M{\set{\Upsilonsub_{ 
		\indexminus},\Upsilonsub_{ 
		\indexplus},\Upsilonsub_{ 
		\indexz}}_\wosetintable} \\
		\hline \hline \rule{0pt}{10pt}stable & \M{\Upsilonsub_{ 
		\indexplus} \vr{1}_\mathbf{x} \exp(-\xi) + \Upsilonsub_{ 
		\indexz} \vr{1}_\mathbf{z} \exp(-D\xi)} & \M{\vec{\vr{u}}(\Lambda;\xi) - \vr{1}\MeanValue{\Lambda\ol{r}(\vec{\vr{u}})}{\vec{x}}{1}} & UM\M{\colon \set{0,-1,-}_\wosetintable} \\
		\hline \rule{0pt}{10pt}centre & \M{-\vr{1} \MeanValue{\Lambda\ol{r}(\vec{\vr{u}})}{\vec{x}}{1}} & \M{\vec{\vr{u}}(\Lambda;\xi) + \left(\vr{1}_\mathbf{y}+(1-D^{-1})\vr{1}_\mathbf{z}\right) \Lambda\ol{r}(\vec{\vr{u}})} & L\M{\colon \set{-1,\:,\:}_\wosetintable} \\
		\hline 
	\end{tabular}
	\vskip\belowdisplayskip\vskip\belowdisplayskip\caption{Approach to the non-hyperbolic singularity \M{\vr{1}} on condition \M{\pd{x} \ol{r}(\vr{1})=0}. The final column indicates class initials and coefficient values, (\M{-}) denoting a strictly negative unknown. Both eigenspaces incorporate nonlinear supplements for \M{\Lambda}-elevation.} 
\label{tab:KP:1nonHyp} \end{table}
\begin{lemma}\label{lem:KP:1} 
	\FirstLine{Any singular terminal asymptote to \M{\vec{\vr{u}}(\Lambda;+\infty)=\vr{1}\ol{X}=\vr{1}}} obeys 
	\begin{equation*}
		\pd{x} \ol{r}(\vr{1}) \leq \pd{z} \ol{r}(\vr{1}) = 0 
	\end{equation*}
	by \cref{eq:PW:U def,eq:PW:ker r}, so its Jacobian eigenvalues are 
	\begin{align*}
		\Omega_{ 
		\indexminus}(\Lambda;1) &\deq -\left.\left(1 - \sqrt{1 + 4\Lambda \pd{x} \ol{r}(\vr{1})} \right)\right/2 \\
		\Omega_{ 
		\indexplus}(\Lambda;1) &\deq -\left.\left(1 + \sqrt{1 + 4\Lambda \pd{x} \ol{r}(\vr{1})} \right)\right/2 \\
		\Omega_{ 
		\indexz}(\Lambda;1) &\deq -D 
	\end{align*}
	Because \M{\vr{1}_{ 
	\indexz}(\Lambda;1)\deq\vr{1}_\mathbf{z}}, approach from phase space interior \M{\st{S}(\Lambda)} requires 
	\begin{equation*}
		\bmodulus{\Upsilonsub_{ 
		\indexminus}(\Lambda;1)}+\bmodulus{\Upsilonsub_{ 
		\indexplus}(\Lambda;1)}>0 
	\end{equation*}
	necessitating 
	\begin{equation*}
		\Lambda \pd{x} \ol{r}(\vr{1}) \geq -1/4 
	\end{equation*}
	Upon violation, \M{\Omega_{ 
	\indexminus}(\Lambda;1)}, \M{\Omega_{ 
	\indexplus}(\Lambda;1)} are complex conjugate, heralding prior regular termination at \M{\olvr{U}\neq\vr{1}}.
	
	\Cref{pro:PW:symmetry} is used to eliminate one unknown in each eigenspace, translation in \M{\xi} conveniently forcing every strictly negative leading coefficient to \M{-1}.
	
	Environment alone determines hyperbolicity at \M{\vr{1}}. The hyperbolic eigenspaces given \M{\pd{x} \ol{r}(\vr{1})<0} are catalogued in \cref{tab:KP:1Hyp}, verifying \M{\Lambda}-elevation of the orbit asymptote \M{\vr{\tilde{v}}(\Lambda;\vec{x})\deq\vec{\vr{u}}(\Lambda;\xi)}. The non-hyperbolic eigenspaces given \M{\pd{x} \ol{r}(\vr{1})=0} are catalogued in \cref{tab:KP:1nonHyp}, the centre eigenspace determining \M{\vec{x}(\Lambda;\xi)} via an orbital integral 
	\begin{equation*}
		\vec{\vr{u}}(\Lambda;\xi)-\vr{1} = -\vr{1} \MeanValue{\Lambda\ol{r}(\vec{\vr{u}})}{\vec{x}}{1} \quad \implies \quad \vec{x}(\Lambda;\xi)-1 = -\MeanValue{\Lambda\ol{r}(\vr{1}\vec{x})}{\vec{x}}{1} 
	\end{equation*}
\end{lemma}
\begin{lemma}\label{lem:KP:1reg} 
	\FirstLine{Any regular terminal asymptote \M{\image{\vec{\vr{u}}}{\closed{\Xi_o}{\Xi}_{\Lambda}}} to \M{\vec{\vr{u}}(\Lambda,\Xi) = \olvr{U}\deq(1,\ol{Y},\ol{Z})^\intercal}} is strictly upper class and linearizes as shown in \cref{fig:KP:Class}. The regular linear asymptote \M{\vec{\vr{v}}(\Lambda;\vec{x})\deq\vec{\vr{u}}(\Lambda;\xi)} is given by \cref{lem:KP:1} subject to 
	\begin{align*}
		&\T{Replace } \pd{x} \ol{r}(\vr{1}) \T{ with } \pd{x} \ol{r}(\olvr{U}) \\
		&\T{Replace } \ol{X} \T{ with } \ \ol{Y} +(1-\ol{Y})\modulus{\sgn(\pd{x} \ol{r}(\vr{1}))} \\
		&\T{Set } \set{\Upsilonsub_{ 
		\indexminus},\Upsilonsub_{ 
		\indexplus},\Upsilonsub_{ 
		\indexz}}_{\suc} \T{ such that } \ \M{\modulus{\Upsilonsub_{ 
		\indexplus}(\Lambda;1)}=1}, \M{\mathop{\arg} \Upsilonsub_{ 
		\indexplus}(\Lambda;1) \in \open{\pi/2}{3\pi/2}}\\
		&\QT{and}\vr{\tilde{v}}(\Lambda;\ol{x}) = \olvr{U} + \pd{\xi} \olvr{u}(\Lambda;\Xi) (\xi-\Xi) + \pd{\xi}\pd{\xi} \olvr{u}(\Lambda;\Xi) (\xi-\Xi)^2/2 + o\lrp{(\Xi-\xi)^2} 
	\end{align*}
	This entails 
	\begin{align*}
		\log \lrprm{-\Upsilonsub_{ 
		\indexplus}(\Lambda;1)} + \Xi \: \Omega_{ 
		\indexplus}(\Lambda;1) &\deq \log\lrprm{\ol{Y}-1} - \log\sqrt{1+4\Lambda\pd{x} \ol{r}\lrp{\olvr{U}}} \\
		\Upsilonsub_{ 
		\indexminus}(\Lambda;1) &\deq -\Upsilonsub_{ 
		\indexplus}(\Lambda;1) \exp\lrprm{-\Xi \sqrt{1+4\Lambda\pd{x} \ol{r}\lrp{\olvr{U}}}} 
	\end{align*}
	unless \M{\Lambda\pd{x} \ol{r}\lrp{\olvr{U}}=-1/4}, when 
	\begin{align*}
		\Upsilonsub_{ 
		\indexplus}(\Lambda;1) &\deq -1 \\
		\Xi \: \Omega_{ 
		\indexplus}(\Lambda;1) &\deq \log\lrprm{\ol{Y}-1} + \log\lrprm{\Xi-2} \\
		\Upsilonsub_{ 
		\indexminus}(\Lambda;1) &\deq 1/(\Xi-2) 
	\end{align*}
	Computation of \M{\Upsilonsub_{ 
	\indexz}(\Lambda;1)} is equally viable, but tediously unedifying. Suffice to say that \M{\ol{Y}=1} is enough to force singular termination \M{\Xi=+\infty}, thence \M{\ol{Z}=1}. In non-hyperbolic environments the orbit must asymptote within the stable eigenspace and \M{\Upsilonsub_{ 
	\indexminus}} is irrelevant.
	
	The limit \M{\vr{\tilde{v}}(\Lambda,+\infty)=\vr{1}} in hyperbolic environments, else \M{\vr{\tilde{v}}(\Lambda,+\infty)=\vr{1}\ol{Y}} assuming \M{\image{\ol{r}}{\image{\vr{\tilde{v}}}{\closed{1}{\ol{X}}_{\Lambda}}}=\set{0}}. This is entirely virtual, however, as \M{\image{\olvr{\vr{u}}}{\open{\Xi}{+\infty}_{\Lambda}}} is formally undefined and not supposed to satisfy anything. Rather, the linear regular asymptote is simply a lucid expression of the Taylor series about \M{\xi=\Xi}. 
\end{lemma}
\begin{corollary}\label{cor:KP:1 smooth} 
	\FirstLine{The compact orbit elevation is differentiable \M{\ol{v}_\mathbf{y}\in\cont{1}{\closed{0}{1}_{\Lambda}}}}. In contrast, neither \M{\pd{x}\ol{v}_\mathbf{z}(\Lambda;1)} nor \M{\pd{x}\pd{x}\ol{v}_\mathbf{y}(\Lambda;1)} need exist. 
\end{corollary}
\begin{corollary}\label{cor:KP:class} 
	\FirstLine{The compact orbit \olst[\Lambda]{V} of any problem is classified} in \cref{tab:KP:class}, following \cref{def:KP:class}. There is no plane wave for problem \M{\Lambda} if its compact orbit is strictly upper class U. Otherwise, the unique plane wave for \M{\Lambda} is the interior \M{\st{V}(\Lambda)} of its compact orbit. 
\end{corollary}
\begin{corollary}\label{cor:KP:lim} 
	\FirstLine{The limiting upper class orbital is a middle class plane wave} 
	\begin{equation*}
		\Upsilonsub_{ 
		\indexminus}(\Lambda;1) \to 0 \QT{as} \ol{v}_\mathbf{y}(\Lambda;1) \to 1 
	\end{equation*}
\end{corollary}

\begin{table}
	\begin{tabular}
		{|c|c|c|c|c|c|} \hline \rule{0pt}{12pt}\textbf{Class} & \M{-4\Lambda\pd{x}{\ol{r}(\vr{1})}} & \M{\ol{v}_\mathbf{y}(\Lambda;1)-1} & \M{1-\pd{x}\ol{v}_\mathbf{y}(\Lambda;1)} & \M{\Lambda\pd{\Lambda}\left(1-\pd{x}\ol{v}_\mathbf{y}(\Lambda;1)\right)\phantom{_{|_{|}}}} \\
		\hline \hline \rule{0pt}{10pt}U strictly & any & regular & \M{1-\Lambda\ol{r}/\left(\ol{v}_\mathbf{y}-1\right)} & \M{-\Lambda\pd{\Lambda}\left(\Lambda\ol{r}/(\ol{v}_\mathbf{y}-1)\right)} \\
		upper & \M{\in\olst{R^+}} & \M{\in\st{R^+}} & \M{\in\set{1-}} & \M{\in\set{0}} \\
		\hline \rule{0pt}{10pt}UM upper & non-hyperbolic & singular & \M{-\Omega_{ 
		\indexplus}(\Lambda;1)-\Lambda\ol{r}/\left(\ol{v}_\mathbf{y}-\ol{x}\right)} & \M{-\Lambda\pd{\Lambda}\left(\Lambda\ol{r}/(\ol{v}_\mathbf{y}-\ol{x})\right)} \\
		middle & \M{\in\set{0}} & \M{\in\set{0}} & \M{\in\set{1-}} & \M{\in\set{0-}} \\
		\hline \rule{0pt}{10pt}M strictly & hyperbolic & singular & \M{-\Omega_{ 
		\indexplus}(\Lambda;1)} & \M{-\Omega_{ 
		\indexplus}\left(1+\Omega_{ 
		\indexplus}\right)/\left(1+2\Omega_{ 
		\indexplus}\right)} \\
		middle & \M{\in\open{0}{1}} & \M{\in\set{0}} & \M{\in\open{1/2}{1}} & \M{\in\st{R^-}} \\
		\hline \rule{0pt}{10pt}LM lower & degenerate & singular & \M{\in\set{-\Omega_{ 
		\indexminus}(\Lambda;1),-\Omega_{ 
		\indexplus}(\Lambda;1)}} & \M{\in\set{-\Lambda\pd{\Lambda}\Omega_{ 
		\indexminus},-\Lambda\pd{\Lambda}\Omega_{ 
		\indexplus}}} \\
		middle & \M{\in\set{1}} & \M{\in\set{0}} & \M{=\set{1/2}} & \M{=\set{+\infty,-\infty}} \\
		\hline \rule{0pt}{10pt}L strictly & hyperbolic & singular & \M{-\Omega_{ 
		\indexminus}(\Lambda;1)} & \M{-\Omega_{ 
		\indexminus}\left(1+\Omega_{ 
		\indexminus}\right)/\left(1+2\Omega_{ 
		\indexminus}\right)} \\
		lower & \M{\in\open{0}{1}} & \M{\in\set{0}} & \M{\in\open{0}{1/2}} & \M{\in\st{R^+}} \\
		\hline \rule{0pt}{10pt}L strictly & non-hyperbolic & singular & \M{-\Omega_{ 
		\indexminus}(\Lambda;1)-\Lambda\pd{x}\ol{r}} & \M{-\Lambda\pd{x}\ol{r}} \\
		lower & \M{\in\set{0}} & \M{\in\set{0}} & \M{\in\set{0+}} & \M{\in\set{0+}} \\
		\hline 
	\end{tabular}
	\vskip\belowdisplayskip\vskip\belowdisplayskip\caption{Compact orbit classification, calculated at \M{(\ol{x}=1)}. The singleton \set{0+} indicates \set{0} entered from \st{R^+} as \M{\ol{x}\to 1}. Likewise \M{\set{0-}\deq\set{0}}, \M{\set{1-}\deq\set{1}} are singletons entered from \st{R^-}. The final column is invalid across class boundaries.} 
\label{tab:KP:class} \end{table}

\section{Interior Properties}\label{sec:IP} 
\begin{remark}
	The heart of this paper contrasts compact orbits as lentor \M{\Lambda} varies. Distinct orbits are contrasted orthogonal to the \M{x}-direction. This ordinate is privileged by monotonicity and fixed domain, which cannot be said for other candidates \M{y}, \M{z} or \M{\xi}. Crucially, contrast proves positive and uniformly continuous with respect to \M{x} in \cref{lem:IP:Bound}. 
\end{remark}
\begin{definition}\label{def:IP:IP} 
	\FirstLine{Interior properties \M{\vr{f}\in\cont{\infty}{\open{0}{1}_{\Lambda}}}} include the interior orbit \M{\st{V}(\Lambda)} and any travelling wave coordinate \M{\xi(\Lambda;x)} by \cref{lem:O:PS} and \cref{lem:O:NADS}. The former is expressed for contrast as 
	\begin{equation*}
		\vr{w}(\Lambda;x) \deq 
		\begin{pmatrix}
			w_\mathbf{y}(\Lambda;x) \\
			w_\mathbf{z}(\Lambda;x) 
		\end{pmatrix}
		\deq 
		\begin{pmatrix}
			(v_\mathbf{y}-x)^2 \\
			v_\mathbf{z}-x 
		\end{pmatrix}
	\end{equation*}
	so \M{\pd{\Lambda} (\pd{x} \ol{w}_\mathbf{y}(\Lambda;1))} exists even when \M{\pd{\Lambda} (\pd{x} \ol{v}_\mathbf{y}(\Lambda;1))} diverges. 
\end{definition}
\begin{definition}\label{def:IP:CO} 
	\FirstLine{The contrast operator \M{\op{d}\colon \cont{\infty}{\open{0}{1}_{\Lambda}} \to \cont{\infty}{\open{0}{1}_{\Lambda}}}} is 
	\begin{equation*}
		\op{d} \vr{f}(\Lambda;x) \deq \vr{f}(\Lambda + \eps;x)-\vr{f}(\Lambda;x) \quad \T{for parameter } \eps \in \st{R^+} 
	\end{equation*}
	Reflecting any concrete integration or experiment, the \M{\eps} limit is always taken prior, as in 
	\begin{equation*}
		\pd{\Lambda+}\olvr{f}(\Lambda;X) \deq \lim\limits_{x\to X} \lim\limits_{\eps \to 0} \frac{\op{d} \vr{f}(\Lambda;x)}{\eps} 
	\end{equation*}
	The contrast operator is right-sided, asymptoting to the right-sided derivative. Only if the contrast operator necessarily agrees with the corresponding left-sided contrast as \M{\epsilon\to0} shall we invoke \M{\pd{\Lambda}\olvr{f}\deq\pd{\Lambda+}\olvr{f}=\pd{\Lambda-}\olvr{f}}. 
\end{definition}
\begin{lemma}\label{lem:IP:CE} 
	\FirstLine{The contrast equation} 
	\begin{equation*}
		\begin{pmatrix}
			-1&\opim{L}{\mathbf{z}}{}{}\\\opim{L}{\mathbf{y}}{}{}&-1 
		\end{pmatrix}
		\Lambda\pd{\Lambda}\vr{w}(\Lambda;x) = 
		\begin{pmatrix}
			0\\r(\vr{u})/\pd{z}r(\vr{u}) 
		\end{pmatrix}
	\end{equation*}
	governs interior orbit contrast. The linear operators \mapdef{\opim{L}{\mathbf{y}}{}{},\opim{L}{\mathbf{z}}{}{}}{\cont{\infty}{\open{0}{1}_{\Lambda}}}{\cont{\infty}{\open{0}{1}_{\Lambda}}} are given by 
	\begin{gather*}
		\opim{L}{\mathbf{y}}{}{} \deq l_{\mathbf{y}}(\Lambda;x) \left(1 + (y-x)\pd{x} \right) \deq \frac{1}{2(y-x)\Lambda \pd{z}r(\vr{u})} \left(1 + (y-x)\pd{x} \right) \\
		\opim{L}{\mathbf{z}}{}{} \deq l_{\mathbf{z}}(\Lambda;x) \left(1 + D^{-1}(y-x)\pd{x} \right) \deq \frac{2(y-x)^2}{(z-x)} \left(1 + D^{-1}(y-x)\pd{x} \right) 
	\end{gather*}
	ignoring \opim{L}{\mathbf{z}}{}{} and \M{w_\mathbf{z}} in any equidiffusive environment \M{(\ol{r};D=1)}. 
\end{lemma}
\begin{proof}
	A check on the uniform validity of limit exchange \M{\pd{\Lambda}(\pd{x}\vr{w})=\pd{x} \pd{\Lambda}\vr{w}}. Applying the contrast operator to Non-Autonomous \cref{lem:O:NADS} yields the contrast equation as \M{\eps \to 0}, on condition that 
	\begin{align*}
		\op{d}w_\mathbf{y}(\Lambda;x) &= o(w_\mathbf{y}(\Lambda;x)) \\
		\op{d}\lrp{\Lambda r(\vr{u})} &= \eps r(\vr{u}) + \Lambda \pd{z} r(\vr{u}) \op{d}w_\mathbf{z} + o\lrp{\op{d}\lrp{\Lambda r(\vr{u})}} 
	\end{align*}
	\Cref{eq:PW:r smooth,eq:PW:r inc z} and \cref{lem:O:PS} readily secure these conditions throughout any \hfill\M{\closed{X_o}{X}_{\Lambda}\subset\open{0}{1}_{\Lambda}}\hfill obeying 
	\newline \M{\norm{\op{d}\olvr{w}}{\closed{X_o}{X}_{\Lambda}}=O(\eps)}. The latter aspect of the contrast equation is demonstrated in \cref{lem:IP:Bound}. 
\end{proof}
\begin{corollary}\label{cor:IP:CEBC} 
	\FirstLine{The contrast equation is reliable on \M{\clopen{0}{1}_{\Lambda}}}, subject to \M{\pd{\Lambda}\olvr{w}(\Lambda;0) = \op{\ol{X}} \vr{0}}. 
\end{corollary}
\begin{definition}\label[notation]{def:IP:Index} 
	\FirstLine{Index notation} deploys 
	\begin{align*}
		\QT{The factor ring of integers} & \st{Z} / N\st{Z} \deq \set{0,1,\dots,N-1}_> \\
		\QT{The non-zero index} & n \in \st{Z} \setminus \set{0} \\
		\QT{The index parity} &\modulus{\mathbf{n}} \deq n \bmod 2 \in \st{Z} / 2\st{Z}\deq\set{0,1}_> \\
		\QT{The index component} &\mathbf{n} \deq n \bmod 2 \in \set{\mathbf{y},\mathbf{z}}_{\succ} \cong \set{0,1}_> 
	\end{align*}
	so even index values refer to \M{\mathbf{y}}-components, odd ones to \M{\mathbf{z}}-components. It is worth defining the constant 
	\begin{equation*}
		L_{\mathbf{n}} \deq \bmodulus{1-\modulus{\mathbf{n}}}+\modulus{\mathbf{n}}\sgn(D-1) = \sgn \lrprm{l_{\mathbf{n}}(\Lambda;x)} \T{ for all }x\in\open{0}{1} 
	\end{equation*}
\end{definition}
\begin{lemma}\label{lem:IP:CO Inv} 
	\FirstLine{The mean value inverses \M{\opim{L}{\mathbf{y}}{-1}{},\opim{L}{\mathbf{z}}{-1}{}}} apply on \M{\closed{X_o}{X}_{\Lambda} \subset \clopen{0}{1}_{\Lambda}} in the sense that 
	\begin{equation*}
		\opim{L}{\mathbf{n}}{-1}{} \opim{L}{\mathbf{n}}{}{} \ol{f}(\Lambda;x) = \frac{\ol{g}_\mathbf{n}(X)\ol{f}(X)-\ol{g}_\mathbf{n}(X_o)\ol{f}(X_o)}{\ol{g}_\mathbf{n}(X)-\ol{g}_\mathbf{n}(X_o)} 
	\end{equation*}
	where the weight is 
	\begin{equation*}
		\ol{g}_\mathbf{n}(\Lambda;\ol{x}) \deq \exp(D^{\modulus{\mathbf{n}}}\xi) 
	\end{equation*}
	and the inverse 
	\begin{equation*}
		\opim{L}{\mathbf{n}}{-1}{} \ol{h}(\Lambda;\ol{x}) \deq \MeanValue[g_\mathbf{n}(\Lambda;x)]{\frac{h(\Lambda;x)}{l_\mathbf{n}(\Lambda;x)}}{X_o}{X} 
	\end{equation*}
	is a mean value provided 
	\begin{equation*}
		\frac{\ol{h}(\Lambda;\ol{x})}{\ol{l}_{\mathbf{n}}(\Lambda;\ol{x})} \in \cont{0}{\closed{X_o}{X}_\Lambda} \\
	\end{equation*}
\end{lemma}
\begin{proof}
	Contrast \cref{lem:IP:CE} expands the substitution 
	\begin{equation*}
		\ol{h}(\Lambda;\ol{x}) = \opim{L}{\mathbf{n}}{}{} \ol{f}(\Lambda;\ol{x}) = \frac{\ol{l}_{\mathbf{n}}(\ol{x})}{\pd{\xi}\ol{g}_\mathbf{n}(\ol{x})} \pd{\xi} \lrp{\ol{g}_\mathbf{n}(\ol{x}) \ol{f}(\ol{x})} 
	\end{equation*}
	\Cref{def:O:MV} readily applies to \M{\opim{L}{\mathbf{n}}{-1}{} \ol{h}(\Lambda;\ol{x})} to show the Lemma. 
\end{proof}
\begin{corollary}\label{cor:IP:IO} 
	\FirstLine{Interior orbit contrast at \M{(\Lambda;\ol{x})\in\clopen{0}{1}_{\Lambda}}} obeys 
	\begin{equation*}
		\Lambda\pd{\Lambda} w_\mathbf{n}(\Lambda;\ol{x}) = \MeanValue[g_\mathbf{n}(x)]{\frac{\Lambda\pd{\Lambda} w_{\prec\mathbf{n}}(x)}{l_{\mathbf{n}}(x)} + \frac{(1-\modulus{\mathbf{n}})}{l_{\mathbf{n}}(x)} \frac{r(\vr{v}(x))}{\pd{z}r(\vr{v}(x))}}{0}{\ol{x}} 
	\end{equation*}
	where the predecessor component is by definition \M{\mathop{\prec}\mathbf{n}\neq\mathbf{n}}. 
\end{corollary}
\begin{lemma}\label{lem:IP:Bound} 
	\FirstLine{Interior orbit contrast is uniformly bound} by 
	\begin{align*}
		\frac{2\Lambda \pd{\Lambda} (\Lambda r)(\image{\vr{v}}{\open{0}{1}_{\Lambda}})}{2+\sgn(D-1)+\sgn(D-1)^2} &\subseteq \open{0}{\norm{\Lambda\ol{r}}{\olst[\Lambda]{S}}} \\
		\frac{\Lambda\image{\pd{\Lambda} w_\mathbf{y}}{\open{0}{1}_\Lambda}}{2+\sgn(D-1)+\sgn(D-1)^2} &\subseteq \open{0}{\norm{\Lambda\ol{r}}{\olst[\Lambda]{S}}^2} \\
	\end{align*}
\end{lemma}
\begin{proof}
	\Cref{lem:O:PS} following \cref{cor:IP:IO} subject to the claim 
	\begin{equation*}
		-\frac{r(\vr{v}(x))}{\pd{z} r(\vr{v}(x))} < \Lambda\pd{\Lambda} w_\mathbf{z}(\Lambda;x) < \frac{r(\vr{v}(x))}{\pd{z}r(\vr{v}(x))} 
	\end{equation*}
	Vacuously true in equidiffusive environments, the claim when \M{D\neq 1} takes a Section's discursion culminating in \cref{cor:UD:result}. 
\end{proof}

\section{Breach Cascades Under Unequal Diffusivity} \label{sec:UD} 
\begin{remark}
	The last Lemma is but the most recent generation in the dynasty of Contrast \cref{lem:IP:CE}. Its ancestry is traced to its origin in this Section, the Contrast Lemma imposing a prerequisite condition for any given bound to be breached. Interpreting this condition itself as a bound, whose breach is prerequisite to the first, commences an impossibly endless induction.
	
	The sequence of breaches dictated by orbit contrast is best considered in reverse, as a cascade of prerequisites. Loosely speaking, breach number one can only happen if the umpteenth breach occurred between the origin and here. Ultimately, the latest breach is impossible unless some prerequisite bound originated in breach. This original breach must be discernible on \M{\open{0}{1}_{\Lambda}}, without recourse to the origin \vr{0}.
	
	This Section formally implements this argument against infinite ascent of a falling sequence or cascade. The essence of the cascade is that its lowest absolute index value is logically final, and vice versa. Lamentably arcane bookkeping is imposed by the vicissitudes of \cref{lem:IP:CE}, obliging four different cascades to cover upper and lower bounds for \M{D<1} and \M{D>1}. Exemplary bound cascades are depicted in \crefrange{fig:UD:Dlt1}{fig:UD:Dgt1}. 
\end{remark}

\begin{figure}
	[htbp] \centering 
	\begin{minipage}
		{0.496 
		\textwidth} \centering 
		\includegraphics[scale=0.5]{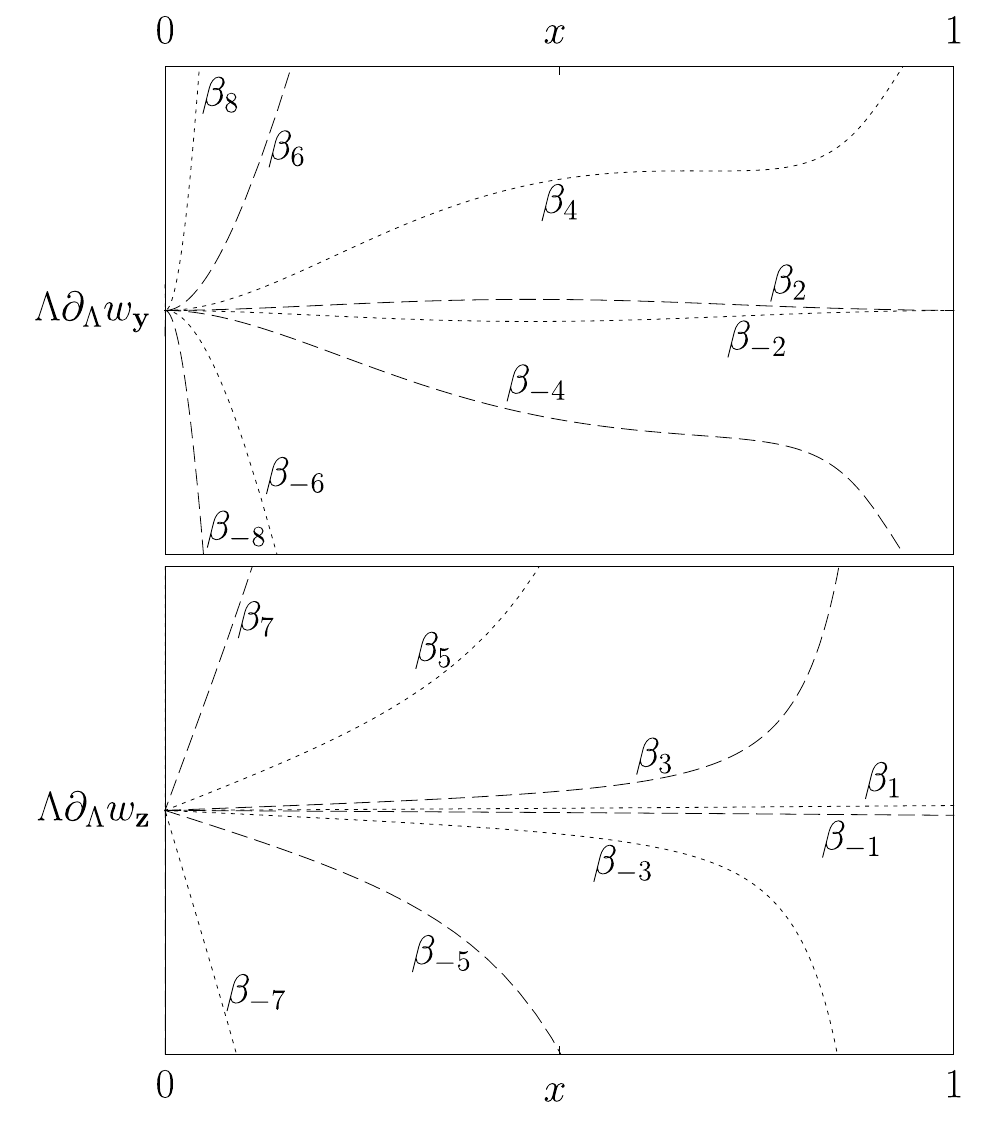} \captionsetup{margin={4pt,32pt}} \caption{Bound cascades for \M{D=1/2}, \M{\Lambda r(\vr{u})=z(1-x)^2/2}. Dashed cascade \M{\CasSeq[-]{\beta}} and dotted cascade \CasSeq[+]{\beta} are labelled to ascend from right to left, alternating between components.} 
	\label{fig:UD:Dlt1} \end{minipage}
	\begin{minipage}
		{0.496 
		\textwidth} \centering 
		\includegraphics[scale=0.5]{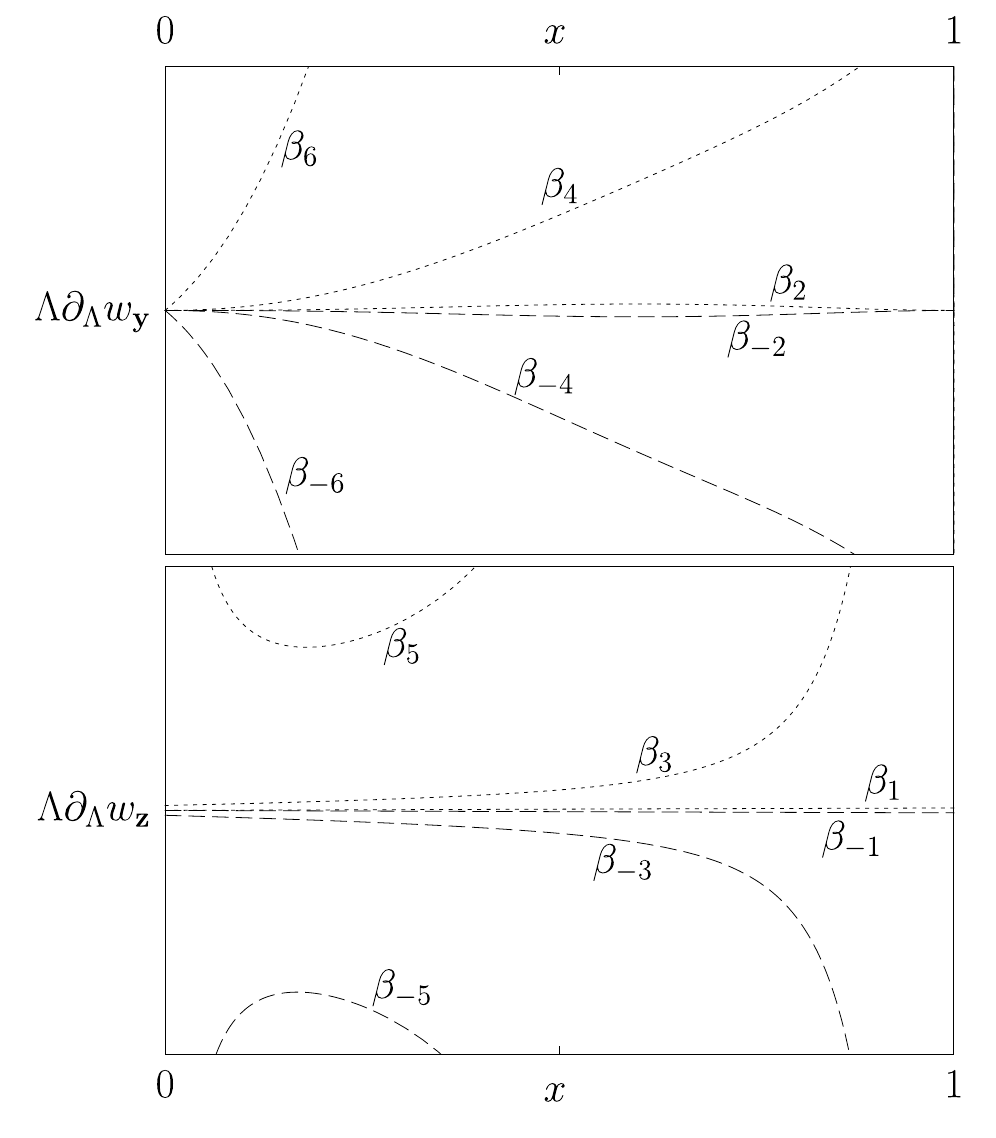} \captionsetup{margin={4pt,32pt}} \caption{Bound cascades for \M{D=2}, \M{\Lambda r(\vr{u})=z^2(1-x)^2/2}. Dashed cascade \M{\CasSeq[-]{\beta}} and dotted cascade \CasSeq[+]{\beta} are labelled to ascend from right to left, hopping between components.} 
	\label{fig:UD:Dgt1} \end{minipage}
\end{figure}

\begin{definition}\label{def:UD:cas} 
	\FirstLine{The cascade sequences \M{\CasSeq[-]{},\CasSeq[+]{}}} are introduced as 
	\begin{equation*}
		\CasSeq{} \deq \setbuilder{\prec^{\modulus{n}}(\pm 1)}{\modulus{n} \in \st{Z} / \CasLen\st{Z}} \cong \set{\CasLen,\ldots,1}_< 
	\end{equation*}
	using Index \cref{def:IP:Index} together with the prerequisite operation 
	\begin{equation*}
		\mathop{\prec}\colon n \mapsto (n+\sgn(n))L_{\mathbf{n}} \in \set{-(n+1),(n+1)} 
	\end{equation*}
	and the cascade lengths \M{\CasLen \in \set{\CasLen[-],\CasLen[+]} \subset \st{Z^+}}. 
\end{definition}
\begin{proposition}\label{pro:UD:cas} 
	\FirstLine{The cascade sequences compute to} 
	\begin{align*}
		D<1\colon\quad &\CasSeq[-]{} = \set{\pm\CasLen[-],\dots,-5,-4,+3,+2,-1}_{\pre} \quad &&\CasSeq[+]{} = \set{\pm\CasLen[+],\dots,+5,+4,-3,-2,+1}_{\pre} \\
		D>1\colon\quad &\CasSeq[-]{} = \set{-\CasLen[-],\dots,-5,-4,-3,-2,-1}_{\pre} \quad &&\CasSeq[+]{} = \set{+\CasLen[+],\dots,+5,+4,+3,+2,+1}_{\pre} 
	\end{align*}
	as compelled by the occurrence of \M{\sgn(z-x)=\sgn(D-1)\deqr L_{\mathbf{1}}} in Contrast \cref{lem:IP:CE}. 
\end{proposition}
\begin{definition}\label{def:UD:Bound} 
	\FirstLine{The bound cascades} \M{\CasSeq{\beta} \subset \cont{\infty}{\open{0}{1}_\Lambda}} are defined recursively as 
	\begin{align*}
		\beta_{\pm1}(\Lambda;x) &\deq \pm \frac{r(\vr{v})}{\pd{z}r(\vr{v})} \\
		\beta_{\mathord{\prec}n}(\Lambda;x) &\deq \opim{L}{\mathbf{n}}{}{} \beta_n(x) +(\modulus{\mathbf{n}}-1) \frac{r(\vr{v})}{\pd{z}r(\vr{v})} 
	\end{align*}
	The original value \M{\beta_n (\Lambda;0)} need not exist. 
\end{definition}
\begin{definition}\label{def:UD:Indic} 
	\FirstLine{The indicator cascades} \M{\CasSeq{\delta} \subset \cont{\infty}{\open{0}{1}_\Lambda}} are defined as 
	\begin{equation*}
		\delta_n(\Lambda;x) \deq \sgn(n)\left(\beta_n(x) - \Lambda\pd{\Lambda} w_\mathbf{n}(x)\right) 
	\end{equation*}
	Negative indices \M{n} refer to lower bounds, positive ones to upper bounds. A negative indicator \M{\delta_n(x)} signals breach. 
\end{definition}
\begin{definition}\label{def:UD:breach} 
	\FirstLine{The breach cascades} \M{\CasSeq{X} \subset \clopen{0}{1} \cup \set{+\infty}} are defined as 
	\begin{equation*}
		X_n(\Lambda) \deq \inf \setbuilder{x \in \open{0}{1}}{\delta_n (\Lambda;x) \le 0} 
	\end{equation*}
	\par A breach \M{X_n=0} is inevitable, and its associated bound is invalid at any \M{x \in \open{0}{1}}.
	
	A breach \M{X_n=\inf\varnothing\deq+\infty} is impossible, and its associated bound is valid at any \M{x \in \open{0}{1}}.
	
	An original breach \M{X_n\in\set{0,+\infty}} is either inevitable or impossible, contingent only on \M{\sgn(\delta_n(x))} immediately departing origin \vr{0}. There are no meaningful prerequisites to this breach. 
\end{definition}
\begin{definition}\label{def:UD:CasLen} 
	\FirstLine{The cascade lengths} \M{\CasLen[-],\CasLen[+] \in \st{Z^+}} are fixed as 
	\begin{equation*}
		\CasLen(\Lambda) \deq \inf \setbuilder{\CasLen \in \st{Z^+}}{\CasSeq{X} \not\subset \open{0}{1}} 
	\end{equation*}
	truncating each cascade at its first original breach. 
\end{definition}
\begin{lemma}\label{lem:UD:Indic} 
	\FirstLine{Indicator cascades \CasSeq{\delta} obey the recurrence relation} 
	\begin{equation*}
		\delta_{\mathord{\prec}n}(\Lambda;x) = L_{\mathbf{n}} \opim{L}{\mathbf{n}}{}{} \delta_n(\Lambda;x) 
	\end{equation*}
	amounting to 
	\begin{gather*}
		\delta_{\pm1}(\Lambda;x) = \frac{r(\vr{u})}{\pd{z}r(\vr{u})} \mp \Lambda\pd{\Lambda} w_\mathbf{z}(\Lambda;x) \\
		\delta_{\mathord{\prec}n}(\Lambda;x) = \left(\sgn(D-1)\opim{L}{\mathbf{z}}{}{}\right)^{\modulus{\mathbf{n}}}\left(\sgn(D-1)\opim{L}{\mathbf{y}}{}{}\opim{L}{\mathbf{z}}{}{}\right)^{(\modulus{n}-\modulus{\mathbf{n}})/2} \delta_{\pm1}(\Lambda;x) 
	\end{gather*}
\end{lemma}
\begin{proof}
	Apply \M{\opim{L}{\mathbf{n}}{}{}} to \cref{def:UD:Indic} using \cref{def:UD:cas,def:UD:Bound} and Contrast \cref{lem:IP:CE}. 
\end{proof}
\begin{lemma}\label{lem:UD:breach} 
	\FirstLine{Breach cascades are strictly decreasing} 
	\begin{equation*}
		X_{\mathord{\prec}n} < X_n 
	\end{equation*}
\end{lemma}
\begin{proof}
	Existence of \M{\mathop{\prec}n} provides a non-empty closed interval \M{\closed{X_o}{X_n} \subset \open{0}{1}}. Combining \cref{lem:IP:CO Inv,lem:UD:Indic} 
	\begin{equation*}
		0 > \frac{-g_\mathbf{n}(\Lambda;X_o)\:\delta_n(\Lambda;X_o)} {g_\mathbf{n}(\Lambda;X_n)-g_\mathbf{n}(\Lambda;X_o)} = \frac{\opim{L}{\mathbf{n}}{-1}{} \delta_{\mathord{\prec}n}(\Lambda;X_n)}{L_{\mathbf{n}}} = \MeanValue[g_\mathbf{n}(\Lambda;x)]{\frac{\delta_{\mathord{\prec}n}(\Lambda;x)}{L_{\mathbf{n}}l_{\mathbf{n}}(\Lambda;x)}}{X_o}{X_n} 
	\end{equation*}
	requiring \M{0>\delta_{\mathord{\prec}n}(\Lambda;x)} at some \M{x \in \closed{X_o}{X_n}}. 
\end{proof}
\begin{corollary}\label{cor:UD:breach} 
	\FirstLine{Any cascade of length \M{\CasLen>1} originated inevitably \M{X_{\CasLen}=0}.} 
\end{corollary}
\begin{lemma}\label{lem:UD:0} 
	\FirstLine{There are no inevitable breaches} as \M{x\to 0}. Formally, given any \M{N\in\st{Z^+}} there is a non-empty open interval \M{\open{0}{X_N}} throughout which 
	\begin{equation*}
		\delta_{\mathord{\prec}\mathord{\prec}n}(\Lambda;x) > \delta_n (\Lambda;x) > 0 \QT{for all} n \in \st{Z} \setminus \set{0} \QT{such that} \modulus{n} < N 
	\end{equation*}
\end{lemma}
\begin{proof}
	\Cref{tab:KP:0} details two possible situations at origin \vr{0}: hyperbolic singularity and non-hyperbolic singularity. This information is combined with \cref{lem:IP:CE,lem:UD:Indic} to compute asymptotic indicators \M{\delta_n(\Lambda;x)} as \M{x \to 0}. The demonstration is not quite overt, since \M{X_N} is not intended infinitesimal. Rather, the asymptotic limit by definition \citep[pp.\,83]{Rudin76} obliges some \M{X_N>0} fulfilling the Lemma.
	
	Near a hyperbolic singularity, invoke the shorthand \M{\Omega_o\deq\Omega_{ 
	\indexminus}(\Lambda;0)>0} and compute 
	\begin{equation*}\label{eq:UD:hyp} 
		\begin{aligned}
			\delta_{\pm1}(\Lambda;x) &= \frac{(1+\Omega_o)(D+2\Omega_o)\pm(1-D)\Omega_o}{(D+\Omega_o)(D+2\Omega_o)}Dx\left(1+O(x)\right) \\
			\delta_{\mathord{\prec}n}(\Lambda;x) &= \left(\frac{4\Omega_o^3(D+\Omega_o)^3}{D^2\bmodulus{D-1}(1+2\Omega_o)}x^2\right)^{\modulus{\mathbf{n}}/2} \left(\frac{(1+2\Omega_o)(D+\Omega_o)}{\Omega_o\modulus{D-1}}\left(1+O(1-x\pd{x})\right)^2\right)^{\modulus{n}/2} \delta_{\pm1}(\Lambda;x) 
		\end{aligned}
	\end{equation*}
	Restricting 
	\begin{equation*}
		\modulus{n} = o\lrp{1/X_N} 
	\end{equation*}
	unity dominates the big \M{O}'s. In which case the indicators are all positive definite, and the factor raised to \M{\modulus{n}} exceeds unity.
	
	Near a non-hyperbolic singularity compute 
	\begin{equation*}\label{eq:UD:nonHyp} 
		\begin{aligned}
			\delta_{\pm1}(\Lambda;x) &= \frac{r(\vr{u})}{\pd{z}r(\vr{u})}\left(1 + O(\pd{z}r(\vr{u}))\right) \\
			\delta_{\mathord{\prec}n}(\Lambda;x) &= \left(\frac{4D\Lambda^3r(\vr{u})^2\pd{z}r(\vr{u})}{\modulus{D-1}}\right)^{\modulus{\mathbf{n}}/2} \left(\frac{D}{\modulus{D-1}\Lambda\pd{z}r(\vr{u})}\left(1+O(r(\vr{u})\pd{x})\right)^2\right)^{\modulus{n}/2} \delta_{\pm1}(\Lambda;x) 
		\end{aligned}
	\end{equation*}
	The indicators for \M{\modulus{n} \geq 2} diverge in the limit, because their associated bounds do. Whence L'H\^{o}pital simplifies the residuals 
	\begin{equation*}
		O\lrp{\frac{\pd{z}r(\vr{u})}{\log(r(\vr{u}))}} = O\lrp{\frac{\pd{z}r(\vr{u})\:\pd{x}}{\pd{x}\log(r(\vr{u}))}} = O(r(\vr{u})\pd{x}) 
	\end{equation*}
	on the centre eigenspace in \cref{tab:KP:0}. Restricting 
	\begin{equation*}
		\modulus{n} = o\lrp{\inf \image{\modulus{\frac{\log r}{\pd{z}r}}}{\image{\vr{v}}{\open{0}{X_N}_{\Lambda}}}} \QT{for interior orbital} \image{\vr{v}}{\open{0}{X_N}_{\Lambda}}\subset\st{V}(\Lambda) 
	\end{equation*}
	unity dominates the big \M{O}'s. In which case the indicators are all positive definite, and the factor raised to \M{\modulus{n}} exceeds unity. 
\end{proof}
\begin{corollary}\label{cor:UD:result} 
	\FirstLine{All breaches are impossible} \M{\CasSeq{X}=\set{X_{\pm1}}=\set{+\infty}}, so 
	\begin{equation*}
		-\frac{r(\vr{u})}{\pd{z}r(\vr{u})} < \Lambda\pd{\Lambda} w_\mathbf{z}(\Lambda;x) < \frac{r(\vr{u})}{\pd{z}r(\vr{u})} \quad \T{throughout } \open{0}{1}_{\Lambda} \T{ whenever } D\ne 1 
	\end{equation*}
\end{corollary}

\section{Compact Properties} \label{sec:CP} 
\begin{remark}
	This Section establishes a compact domain \M{\set{(\Lambda;x)}} sufficient to house every plane wave. The final Lemma bridges the domain interior -- where the Contrast Lemma operates -- to the domain boundary, taking the final step in our proof. 
\end{remark}
\begin{definition}\label{def:CP:CED} 
	\FirstLine{The environment domain} is 
	\begin{equation*}
		\olst{X}(\ol{r};D) \deq \closed{0}{\ENorm^{-1}} \times \closed{0}{1} 
	\end{equation*}
	where the environment norm on \M{\cont{\infty}{\olst{U}}} is 
	\begin{equation*}
		\ENorm \deq 2\int_{0}^{1} r\lrp{\vr{1}-\vr{1}_\mathbf{x}(1-x)-\vr{1}_\mathbf{z}(1-x)^{\max(D,1)}} \dd{x} 
	\end{equation*}
	This is a compact and perfect topological space: every point is a limit point \M{(\ol{\Lambda};\ol{x}) \in \olst{X}(\ol{r};D)}, not all of which are interior \M{(\Lambda;x) \in \st{X}(r;D)}. 
\end{definition}
\begin{definition}\label{def:CP:CP} 
	\FirstLine{A compact property} \M{\olvr{f}\in\cont{0}{\olst{X}(\ol{r};D)}} is an interior property \M{\vr{f}\in\cont{\infty}{\open{0}{1}_{\Lambda}}} compactified according to 
	\begin{equation*}
		\olvr{f}(\ol{\Lambda};\ol{x}) \deq \lim\limits_{(\Lambda;x) \to (\ol{\Lambda};\ol{x})} \vr{f}(\Lambda;x) 
	\end{equation*}
	Existence and uniqueness of the limit is emphatically demanded by uniform continuity of \M{\olvr{f}\in\cont{0}{\olst{X}(\ol{r};D)}}. 
\end{definition}
\begin{proposition}\label{pro:CP:spurious} 
	\FirstLine{The spurious orbit \M{\olst[0]{V}}} obeying 
	\begin{align*}
		\olvr{v}(0;\ol{x}) &\deq \vr{1}\ol{x} \\
		\Lambda \pd{\Lambda+} \olvr{v}(0;\ol{x}) &\deq \Lambda \ol{r}(\vr{1}\ol{x}) \left(\vr{1}_\mathbf{y}+\left(1-D^{-1}\right)\vr{1}_\mathbf{z}\right) 
	\end{align*}
	is the unique orbit limit as \M{\Lambda\to0}. It is the only spurious compact orbit on the environment domain. In the vicinity of the limit, each compact orbit \M{\olst[\Lambda]{V}} is strictly lower class, containing the plane wave of \M{\Lambda}. 
\end{proposition}
\begin{proof}
	Non-Autonomous \cref{lem:O:NADS} has the unique asymptotic solution 
	\begin{equation*}
		\olvr{v}(\Lambda;x) = \olvr{v}(0;\ol{x}) + \Lambda\pd{\Lambda+} \olvr{v}(0;\ol{x}) + O\lrp{\Lambda^2\ol{r}(\olvr{u})} 
	\end{equation*}
	The limit is not orbital, because it does not support a travelling wave coordinate for Autonomous \cref{def:PW:ADS}. This failing is peculiar to \M{\ol{\Lambda}=0}. For \M{\Lambda>0}, the asymptotic solution is verifiably orbital, of strictly lower class as \M{\Lambda\pd{\Lambda} \olvr{v}(\Lambda;1)=O(\ol{r}(\vr{1}))=0} and \M{\pd{x} \olvr{v}(\Lambda;1)=\vr{1}-O(\Lambda^2)}. 
\end{proof}
\begin{lemma}\label{lem:CP:CO} 
	\FirstLine{The family of interior orbits \M{\image{\st{V}}{\open{0}{\ENorm^{-1}}}} compactifies to the environment orbit \M{\olst{V} \in \cont{0}{\olst{X}(\ol{r};D)}}}. This surface in four-dimensional space naturally decomposes into the environment domain \M{\olst{X}(\ol{r};D)} and the environment orbit projection \M{\op{\ol{X}}\olvr{v}(\ol{\Lambda};\ol{x})}. It inhabits the environment phase space 
	\begin{equation*}
		\olst{V}\subseteq\olst{S} \deq \olst{X}(\ol{r};D) \times \closed{\max(\ol{x},\ol{z})}{\:\ol{x}+\ol{\Lambda}\norm{\ol{r}}{\olst{S}}} \times \closed{\ol{x}\min(D,1)}{\:\ol{x}\max(D,1)} 
	\end{equation*}
	where \M{\norm{\ol{r}}{\olst{S}}\deq\norm{\ol{r}}{\olst[\Lambda]{S}}} independently of \M{\Lambda}, using \cref{eq:PW:U def}. 
\end{lemma}
\begin{proof}
	Following \cref{pro:CP:spurious}, the limits demanded by \cref{def:CP:CP} are only insecure when \M{\ol{\Lambda}>0}. Contrast \Cref{lem:IP:Bound} states that \M{\vr{w}(\Lambda;x)} converges uniformly to \M{\vr{w}(\ol{\Lambda};x)} on \M{\open{0}{1}_\Lambda} as \M{\Lambda\to\ol{\Lambda}>0}. Which suffices to invoke the uniform limit theorem \citep[pp.\,149]{Rudin76}, guaranteeing a unique, uniformly continuous limit \M{\vr{w}(\Lambda;x)\to\olvr{w}(\ol{\Lambda};\ol{x}) \in \cont{0}{\closed{0}{1}_{\ol{\Lambda}}}} as \M{\Lambda\to\ol{\Lambda}>0}. Finally, the square root of any uniformly continuous positive function is unique and uniformly continuous. 
\end{proof}
\begin{corollary}\label{cor:CP:r} 
	\FirstLine{The reaction term is continuously differentiable \M{\ol{\Lambda}\ol{r}(\olvr{u})\in\cont{1}{\olst{X}(\ol{r};D)}}} as \M{\ol{\Lambda}\ol{r}(\olvr{u})=0} whenever \M{\ol{x}=1}. 
\end{corollary}
\begin{lemma}\label{lem:CP:PW} 
	\FirstLine{Every plane wave supported by environment \M{(\ol{r},D)} is interior to \olst{V}} and bound to obey 
	\begin{equation*}
		1-(1-\ol{x})^{\min(D,1)} \leq \ol{v}_{\mathbf{z}}(\Lambda;\ol{x}) \leq 1-(1-\ol{x})^{\max(D,1)} 
	\end{equation*}
\end{lemma}
\begin{proof}
	The bounds on \M{\ol{v}_{\mathbf{z}}} are vacuous in equidiffusive environments, so take \M{D\ne1} and constant \M{\Delta \in \set{D,1}}. Define the indicator 
	\begin{equation*}
		\delta(\Lambda;x) \deq (D+1-2\Delta)\left(v_{\mathbf{z}}(\Lambda;x) - 1+(1-x)^{\Delta}\right) \in \cont{\infty}{\open{0}{1}_{\Lambda}} 
	\end{equation*}
	and the open set on which the quoted bounds on \M{\ol{v}_{\mathbf{z}}} are violated 
	\begin{equation*}
		\st{D}_{\Lambda} \deq \image{\delta^{-1}}{\st{R^-}} 
	\end{equation*}
	The plane wave projection \M{\op{X}\vr{v}(\Lambda;x)} is governed by Non-Autonomous \cref{lem:O:NADS} in the form 
	\begin{equation*}
		\left(D+(v_{\mathbf{y}}-x)\pd{x}\right)\delta(\Lambda;x) = (D+1-2\Delta)\left(D(1-x)^{\Delta} - \Delta (v_{\mathbf{y}}-x) (1-x)^{(\Delta-1)} - D (1-v_{\mathbf{y}})\right) 
	\end{equation*}
	and Boundary Constraints~\labelcref{def:PW:BC0,def:PW:BC1} dictate \M{\image{\ol{\delta}}{\set{\Lambda}\times\set{0,1}} = \set{0}}. Whereupon the mean value theorem demands a stationary point \M{x\in\st{D}_{\Lambda}\neq\varnothing} impossibly governed by 
	\begin{equation*}
		0 > D\delta(\Lambda;x) = (D+1-2\Delta)\left(D(1-x)^{\Delta} - \Delta (y-x) (1-x)^{(\Delta-1)} - D (1-y)\right) > 0 
	\end{equation*}
	following \cref{lem:O:PS} with \M{y<1} throughout a plane wave. Therefore \M{\st{D}_{\Lambda}=\varnothing}, bounding \M{\ol{v}_{\mathbf{z}}(\Lambda;\ol{x})} as claimed.
	
	To ensure plane wave \M{{\st{V}(\Lambda)} \in \cont{0}{\st{X}(r;D)}}, integrate Non-Autonomous \cref{lem:O:NADS} 
	\begin{equation*}
		1=\ol{v}_\mathbf{y}(\Lambda;1)^2 = 2 \Lambda \int_{0}^{1} r(\vr{u}) \dd{x} \: + \: 2 \int_{0}^{1} x\pd{x} v_{\mathbf{y}}(\Lambda;x) \dd{x} > \Lambda \ENorm 
	\end{equation*}
	using the upper bound on \M{\ol{z}(\Lambda;\xi)=\ol{v}_{\mathbf{z}}(\Lambda;\ol{x})} and the monotonies of \cref{eq:PW:r inc z} and \cref{lem:O:PS}. 
\end{proof}
\begin{corollary}\label{cor:CP:slowest} 
	\FirstLine{The slowest compact orbit \M{\olst{V}(\ENorm^{-1})} is strictly upper class}. 
\end{corollary}
\begin{lemma}\label{lem:CP:COC} 
	\FirstLine{Compact orbit contrast} obeys 
	\begin{multline*}
		\op{d} \lrp{ \left(\ol{v}_\mathbf{y}(\Lambda;1)-1\right) - (1-x) \left(\pd{x} \ol{v}_{\mathbf{y}}(\Lambda;1)-1\right) + \left(1-x\right) O\lrp{\pd{x} \ol{v}_{\mathbf{y}}(\Lambda;1)-\pd{x} v_{\mathbf{y}}(\Lambda;\mathring{x})}}^2 \\
		= \eps \pd{\Lambda} w_\mathbf{y}(\mathring{\Lambda};x) \in \open{0}{ \eps\left(2+\sgn(D-1)+\sgn(D-1)^2\right) \mathring{\Lambda}\norm{\ol{r}}{\olst{S}}^2} 
	\end{multline*}
	where \M{\mathring{x}\in\open{x}{1}\subset\open{0}{1}} and \M{\mathring{\Lambda}\in\open{\Lambda}{\Lambda+\eps}\subset\open{0}{\ENorm^{-1}}}. The big \M{O} (reported exactly) shrinks to zero as \M{x\to1}, becoming insignificant unless \M{\op{d} \lrprm{\left(1-(1-x) \pd{x}\right) \sqrt{\ol{w}_\mathbf{y}(\Lambda;1)}}^2=o\lrp{(1-x)^2}}. 
\end{lemma}
\begin{proof}
	Contrast Operator \labelcref{def:IP:CO} -- which is limitless -- simply computes the difference of two functions here sufficiently smooth as follows. The left hand side applies the mean value theorem (in \M{x}) to \M{\op{d} \ol{v}_\mathbf{y}\in\cont{1}{\closed{0}{1}_{\Lambda}}} by \cref{cor:KP:1 smooth}. The right hand side applies the mean value theorem (in \M{\Lambda}) to \M{\op{d} w_\mathbf{y} \in\cont{1}{\st{X}(r;D)}} using \cref{lem:IP:Bound}. 
\end{proof}
\begin{corollary}\label{cor:CP:UCU} 
	\FirstLine{Contrasting two upper class (U or UM) orbits} gives a strictly positive result 
	\begin{equation*}
		\pd{\Lambda+} \ol{w}_\mathbf{y}(\Lambda;1) \in \opclo{0}{\left(2+\sgn(D-1)+\sgn(D-1)^2\right) \Lambda \norm{\ol{r}}{\olst{S}}^2} 
	\end{equation*}
\end{corollary}
\begin{proof}
	Retaining residuals in \cref{lem:CP:COC} using \cref{tab:KP:class} yields 
	\begin{equation*}
		\left(\frac{\op{d} \ol{w}_\mathbf{y}(\Lambda;1)}{\eps\sqrt{\ol{w}_\mathbf{y}(\Lambda;1)}}\left(\sqrt{\ol{w}_\mathbf{y}(\Lambda;1)}+(1-x)\right) - O\lrp{2\frac{\op{d}(\Lambda r(\mathring{\vr{v}}))}{\eps}}(1-x)\right) \in \open{0}{\left(2+\sgn(D-1)+\sgn(D-1)^2\right) \mathring{\Lambda}\norm{\ol{r}}{\olst{S}}^2} 
	\end{equation*}
	where \M{\Lambda r(\mathring{\vr{v}}) \deq \Lambda r(\vr{v}(\Lambda;\mathring{x}))} is strictly increasing in \M{\Lambda} by \cref{lem:IP:Bound}. Apply the mean value theorem (in \M{\Lambda}) to the left hand side prior to taking limits (\M{x\to1} followed by \M{\eps\to0}) to see that \M{\pd{\Lambda+} \ol{w}_\mathbf{y}(\Lambda;1)>0} strictly. Orientation is material to ensuring only upper class orbits are constrasted, the limit being the right-sided derivative in \cref{def:IP:CO}. 
	
	Incidentally, this result is obvious for strictly upper class orbits because weights \M{g_{\mathbf{n}}} never diverge in this case, securing their associated mean values throughout the compact domain -- that is to say \crefrange{lem:IP:CO Inv}{lem:IP:Bound} apply right across \M{\closed{0}{1}_{\Lambda}}. 
\end{proof}
\begin{corollary}\label{cor:CP:MCM} 
	\FirstLine{There is at most one middle class (UM, M or LM) orbit in \olst{V}}. 
\end{corollary}
\begin{proof}
	If the environment is terminally non-hyperbolic (\M{\pd{x} \ol{r}(\vr{1})=0}) any middle class orbits can only be upper middle class (UM), contrasted with each other to give \M{\op{d} w_\mathbf{y}(\Lambda;1)=0} in violation of \cref{cor:CP:UCU}. Assuming a terminally hyperbolic environment (\M{\pd{x} \ol{r}(\vr{1})<0}) contrast any two middle class (M or LM) orbits according to \cref{cor:KP:class} 
	\begin{equation*}
		\op{d} w_\mathbf{y}(\Lambda;x) = \op{d} \lrp{\Omega_{ 
		\indexplus}(\Lambda;1) + o(1)}^2(1-x)^2 < 0 
	\end{equation*}
	violating \cref{lem:CP:COC} as \M{x\to1}. 
\end{proof}

\section{Conclusion}\label{sec:Con} 
\begin{result}\label{thm:CP:pweuc} 
	A single critical lentor \M{\Lambda_*(\ol{r};D)\in\open{0}{\ENorm^{-1}}} is attributable to any given environment \M{(\ol{r};D)}, where 
	\begin{equation*}
		\ENorm \deq 2\int_{0}^{1} r\lrp{\vr{1}-\vr{1}_\mathbf{x}(1-x)-\vr{1}_\mathbf{z}(1-x)^{\max(D,1)}} \dd{x} 
	\end{equation*}
	Problem \M{\Lambda\in\opclo{0}{\Lambda_*}} possesses a unique plane wave in this environment, while problem \M{\Lambda\in\open{\Lambda_*}{\infty}} possesses none. The slowest plane wave \M{\st{V}(\Lambda_*)} is invariably middle class \M{\pd{x}\ol{v}_{\mathbf{y}}^{}(\Lambda;1) \leq 1/2}, all other plane waves \image{\st{V}}{\open{0}{\Lambda_*}} strictly lower class \M{\pd{x}\ol{v}_{\mathbf{y}}^{}(\Lambda;1) > 1/2}.
	
	\hfill The\hfill square\hfill elevation \hfill\M{\ol{w}_\mathbf{y}(\ol{\Lambda};\ol{x})\deq\left(\ol{v}_\mathbf{y}(\ol{\Lambda};\ol{x})-\ol{x}\right)^2}\hfill is\hfill continuously\hfill differentiable\hfill except\hfill at\hfill the\hfill point 
	\newline \M{(\Lambda_*;1)\in\olst{X}(\ol{r};D)\deq\closed{0}{\ENorm^{-1}}\times\closed{0}{1}}. It is monotonically increasing in \ol{\Lambda} and Lipschitz continuous as 
	\begin{equation*}
		\image{\pd{\Lambda-}\ol{w}_\mathbf{y}}{\olst{X}(\ol{r};D)}\cup\image{\pd{\Lambda+}\ol{w}_\mathbf{y}}{\olst{X}(\ol{r};D)} \subseteq \closed{0}{\left(2+\sgn(D-1)+\sgn(D-1)^2\right) \ol{\Lambda} \norm{\ol{r}}{\olst{S}(\ol{\Lambda})}^2} 
	\end{equation*}
\end{result}
\begin{proof}
	In essence \cref{lem:KP:1,lem:KP:1reg} describe termination, \cref{lem:IP:Bound} governs all prior behaviour, and \cref{lem:CP:COC} bridges the two. \Cref{lem:CP:PW} just ensures all plane waves are captured. Assuming the prequel, the Theorem is ultimately a simple combination of \Cref{cor:CP:slowest,cor:CP:UCU,cor:CP:MCM,cor:KP:lim}. 
\end{proof}
\begin{remark}
	\Cref{lem:IP:Bound} portrays two contrasted orbits as a zip fastener starting at \vr{0}. Taking the limit \M{\eps\to0} is like closing the zip. The end of the zip fastens awkwardly, needing help from the virtual constructions of \cref{lem:KP:1reg} and the bridge of \cref{lem:CP:COC} to close securely. This is hardly surprising, for closing the zip unites orbits even across class boundaries featuring irreconcilable differences at termination. Apart from degenerate hyperbolic termination, it seems hard to believe that upper class and lower class orbits coalesce at the critical lentor, as the two span a step change in leading eigenvalue at termination. To quell the associated notion that degenerate termination is to be expected in hyperbolic systems, here is a staightforward result. 
\end{remark}
\begin{corollary}\label{cor:Con:Nondeg} 
	In any hyperbolic environment all plane waves \M{\image{\st{V}}{\open{0}{\Lambda_*}}} are strictly lower class \M{\pd{x}\ol{v}_{\mathbf{y}}^{}(\Lambda;1) > 1/2} except the slowest one \M{\st{V}(\Lambda_*)}. For every environment in which the slowest plane wave is lower middle class \M{\pd{x}\ol{v}_{\mathbf{y}}^{}(\Lambda_*;1) = 1/2} there are countless environments in which \M{\st{V}(\Lambda_*)} is upper middle class \M{\pd{x}\ol{v}_{\mathbf{y}}^{}(\Lambda_*;1) > 1/2}. The latter lack any lower middle class at all, instead featuring a discontinuity in \M{\pd{x}\ol{v}_{\mathbf{y}}^{}(\Lambda;1)} at \M{\Lambda=\Lambda_*}. 
\end{corollary}
\begin{proof}
	Fix the assumed environment as \M{(\ol{r}_{h},D)}, requiring \M{\ENorm[h]<1/\Lambda_*=-4D\pd{x}\ol{r}_{h}(\vr{1})}. Now take any \M{\ol{r}_{nh}} abiding \crefrange{eq:PW:U def}{eq:PW:r inc z} with \M{\pd{x}\ol{r}_{nh}(\vr{1})=0}. Then environment \M{(\ol{r}_{h}+\ol{r}_{nh},D)} lacks a lower middle class whenever \M{\left(\ENorm[nh]+\ENorm[h]\right)>-4D\pd{x}\ol{r}_{h}(\vr{1})} by \cref{lem:CP:PW}. 
\end{proof}
\begin{remark}
	Finally, a natural and well-known extension to cutoff reactions, which proceed only when product is sufficiently abundant. This is extremely common in modelling flames, where an ignition temperature \M{\breve{\mathsf{X}}} serves to ameliorate the notorious cold boundary problem. 
\end{remark}
\begin{corollary}\label{cor:Con:Cutoff} 
	Consider the original physical problem set out the Introduction, altered to incorporate a reaction cutoff 
	\begin{equation*}
		\breve{\mathsf{r}}(\breve{\mathsf{x}},\breve{\mathsf{z}})=0 \QT{for all} \breve{\mathsf{x}}\leq \breve{\mathsf{X}} 
	\end{equation*}
	This problem supports just one plane wave travelling at a characteristic wavespeed \M{V=V_\dag<V_*(1-\breve{\mathsf{X}})^{-1/2}}. 
\end{corollary}
\begin{proof}
	Introducing 
	\begin{equation*}
		\ol{Y}\deq 1/(1-\breve{\mathsf{X}}) 
	\end{equation*}
	adopt the scaled dynamical system formulation 
	\begin{gather*}
		\vr{u} \deq (x,y,z)^{\intercal} \deq \ol{Y} \left((1 - \breve{\mathsf{x}}), (1 - \pd{\xi}\breve{\mathsf{x}} - \breve{\mathsf{x}}),\breve{\mathsf{z}} \right)^{\intercal} \\
		\Lambda r(\vr{u}) \deq \ol{Y} \mathop{V^{-2}}\breve{\mathsf{r}}(\breve{\mathsf{x}},\breve{\mathsf{z}}) 
	\end{gather*}
	which cuts off reaction for \M{x\geq1}. This is just the original problem already studied, except now the Terminal Boundary Constraint is \M{\ol{\vr{u}}(\Lambda;+\infty)=\vr{1}\ol{Y}}. Integrating backwards, the general solution post cutoff is 
	\begin{equation*}
		\olvr{u}(\Lambda;\xi)= \vr{1}\ol{Y} + \vr{1}_{\mathbf{x}}\left(1-\ol{Y}\right)\exp\lrprm{\Xi-\xi} + \vr{1}_{\mathbf{z}}\left(\ol{Z}-\ol{Y}\right)\exp\lrprm{D\Xi-D\xi} 
	\end{equation*}
	which terminates a plane wave if and only if it connects with a strictly upper class orbit of the original problem at \M{\olvr{u}=(1,\ol{Y},\ol{Z})}. Towards termination, the plane wave resembles the upper class orbits depicted in \cref{fig:KP:Class}, where the virtual (grey) region is now actual and gathers any orbit into the solid, horizontal, non-hyperbolic trajectory. \Cref{lem:CP:COC} insists \M{\pd{\Lambda+}\ol{Y}>0} amongst the upper class, thus allowing precisely one plane wave at some \M{\Lambda=\Lambda_{\dag}>\Lambda_*}. 
\end{proof}
\begin{remark}
	In closing, we should reiterate that the Conclusion presented here is by no means novel \citep{Marion85,Volpert94}, but the method of proof is. As a fairly tangible application of basic methods it hopes be of some pedagogical value regarding the limited but extremely important class of problems it covers. Further, the simple ideas presented here -- orbit contrast and breach cascades -- may be useful in understanding other dynamical systems quite unrelated to this work. Finally, one may spare a thought for the relation of this approach to the Leray-Schauder method. In this regard, the reliance of breach cascades on higher derivatives of the reaction function \M{\ol{r} \in \cont{\infty}{\olst{U}}} seems difficult to reconcile. Unnecessary in equidiffusive \M{(D=1)} environments, smoothness is absolutely essential to the breach cascades employed outside this special case. The Leray-Schauder method makes no such demand in any case. 
\end{remark}
\bibliographystyle{plainnat} 
\bibliography{main} \end{document}